\documentclass[twoside,11pt,reqno]{amsart}

\usepackage[margin=1.1in]{geometry}
\usepackage{amsmath} 
\usepackage{amssymb} 
\usepackage{amsopn}
\usepackage{amsthm} 
\usepackage{amsfonts} 
\usepackage{mathrsfs}
\usepackage{stmaryrd}
\usepackage{colonequals}
\usepackage{tikz-cd}
\usepackage{enumitem}

\usepackage{cite}

\usepackage{verbatim}

\usepackage[pdfpagelabels,pdftex]{hyperref}

\usepackage{setspace}
\usepackage{amssymb}
\usepackage{euscript}
\usepackage{cite}
\usepackage[all]{xy}
\usepackage{tabularx}

\theoremstyle{plain}
\usepackage{amsfonts}
\usepackage{graphicx}
\usepackage{amsmath}

\hypersetup{
  pdftitle={},
  pdfauthor={},
  pdfsubject={},
  pdfkeywords={},
  colorlinks=false,
  breaklinks=true,
  bookmarksopen=true,
  bookmarksnumbered=true,
  pdfpagemode=UseOutlines,
  plainpages=false}

\DeclareMathOperator{\GL}{GL}
\DeclareMathOperator{\SL}{SL}
\DeclareMathOperator\ord{ord}

\DeclareMathOperator\Gal{Gal}

\DeclareMathOperator\Ind{Ind}
\DeclareMathOperator\Inf{Inf}

\DeclareMathOperator\Tr{Tr}

\DeclareMathOperator\Aut{Aut}

\newcommand{\sm}{{\,\smallsetminus\,}}

\newcommand\from{\colon}
\newcommand\bG{\mathbb G}
\newcommand\bQ{\mathbb Q}
\newcommand\bU{\mathbb U}
\newcommand\QQ{\mathbb Q}

\newcommand\bR{\mathbb R}
\newcommand\bT{\mathbb T}

\newcommand\bW{\mathbb W}
\newcommand\FF{\mathbb F}
\newcommand\cA{\mathcal A}

\newcommand\bZ{\mathbb Z}
\newcommand\bK{\mathbb K}

\newcommand\cT{\mathcal T}
\newcommand\sB{\mathscr{B}}
\newcommand\fX{\mathfrak{X}}
\newcommand\fp{\mathfrak{p}}
\newcommand\sA{\mathscr{A}}

\DeclareMathOperator\height{ht}
\newcommand\bN{\mathbb N}
\newcommand\bX{\mathbb X}
\newcommand\bB{\mathbb B}

\newcommand\bx{\mathbf{x}}
\newcommand\bH{\mathbb{H}}
\newcommand\rar{\rightarrow}

\newcommand\tar{\twoheadrightarrow}
\newcommand\cool{\overline{\bQ}_{\ell}}

\newcommand\cO{\mathcal O}

\makeatletter
\newtheorem*{rep@theorem}{\rep@title}
\newcommand{\newreptheorem}[2]{%
\newenvironment{rep#1}[1]{%
 \def\rep@title{#2 \ref{##1}}%
 \begin{rep@theorem}}%
 {\end{rep@theorem}}}
\makeatother

\newtheorem{thm}{Theorem}[section]
\newtheorem*{thm*}{Theorem}

\newtheorem{prop}[thm]{Proposition}

\newtheorem{cor}[thm]{Corollary}

\newtheorem*{cor*}{Corollary}
\newtheorem{lm}[thm]{Lemma}

\newtheorem{theorem}[thm]{Theorem}
\newtheorem{lemma}[thm]{Lemma}
\newtheorem{corollary}[thm]{Corollary}
\newtheorem{proposition}[thm]{Proposition}
\newtheorem{proposition-definition}[thm]{Proposition-Definition}
\newtheorem*{conj*}{Conjecture}

\theoremstyle{remark}
\theoremstyle{remark}\newtheorem*{claim*}{Claim}

\newreptheorem{thm}{Theorem}
\newreptheorem{prop}{Proposition}
\newreptheorem{cor}{Corollary}

\theoremstyle{definition}
\newtheorem{Def}[thm]{Definition}

\newtheorem{definition}[thm]{Definition}

\theoremstyle{remark}
\newtheorem{rem}[thm]{Remark}

\newenvironment{pro*}[1][Proof]{{\it{#1:}} }{}

\newcounter{absatzcounter}[section]
\numberwithin{equation}{section}

\begin{document}

\title{Cohomological representations of parahoric subgroups}
\author{Charlotte Chan and Alexander Ivanov}
\address{Department of Mathematics \\
Princeton University \\
Fine Hall, Washington Road \\
Princeton, NJ 08544-1000 USA}
\email{charchan@math.princeton.edu}
\address{Mathematisches Institut \\ Universit\"at Bonn \\ Endenicher Allee 60 \\ 53115 Bonn, Germany}
\email{ivanov@math.uni-bonn.de}

\maketitle

\begin{abstract}
We generalize a cohomological construction of representations due to Lusztig from the hyperspecial case to  arbitrary parahoric subgroups of a reductive group over a local field which splits over an unramified extension. We compute the character of these representations on certain very regular elements.
\end{abstract}


\section{Introduction}

Let $k$ be a non-archimedean local field with finite residue field. Let $G$ be a reductive group over $k$, and let $T \subseteq G$ be a maximal torus defined over $k$ and split over an unramified extension of $k$. Let $P$ be a parahoric model of $G$, defined over the integers $\cO_k$. We denote the schematic closure of $T$ in $P$ again by $T$. We will construct and study a tower of varieties over an algebraic closure of the residue field $\FF_q$ of $k$ whose cohomology realizes interesting representations of $P(\cO_k)$ parametrized by characters of $T(\cO_k)$. This construction generalizes classical Deligne--Lusztig theory \cite{DeligneL_76} (for reductive groups over finite fields), as well as the work of Lusztig \cite{Lusztig_04} and Stasinski \cite{Stasinski_09} (for reductive groups over henselian rings). Further, we give an explicit formula for the character on certain very regular elements, generalizing a special case of the character formula for representations of reductive groups over finite fields \cite[Theorem 4.2]{DeligneL_76}.

More precisely, choose a Borel subgroup of $G$ containing $T$ (defined over some unramified extension of $k$) with unipotent radical $U$. Fix a Moy--Prasad filtration quotient $\bG$ of $P$, regarded as a smooth affine group scheme of finite type over $\FF_q$. As such, one has a Frobenius $\sigma \from \bG \to \bG$ and the corresponding Lang map $\bG \to \bG$, $g \mapsto g^{-1} \sigma(g)$. In $\bG$ we have the subgroups $\bT$ and $\bU$, corresponding to the closures of $T$ and $U$ in $P$.  Consider the subscheme $S_{T,U} \subset \bG$ defined as the preimage of $\bU$ under the Lang map. By construction, $S_{T,U}$ has a natural action of $P(\cO_k) \times T(\cO_k)$ given by left and right multiplication. For a smooth character $\theta \colon T(\cO_k) \rar \cool^\times$, we define $R_{T,U}^\theta$ to be the $\theta$-isotypic component of the alternating sum of the cohomology groups of $S_{T,U}$ with $\cool$-coefficients. This is a virtual $P(\cO_k)$-representation.

\begin{thm}[cf.\ Corollary \ref{cor:indep_of_U_irreduciblity}]\label{thm:1}
If $\theta$ is sufficiently generic, then $R_{T,U}^\theta$ is independent of the choice of $U$. Moreover, if the stabilizer of $\theta$ in the Weyl group of the special fiber of $P$ is trivial, then $\pm R_{T,U}^\theta$ is an irreducible representation of $P(\cO_k)$.
\end{thm}
The proof of Theorem \ref{thm:1} mainly follows the original method of Lusztig \cite{Lusztig_04}, which treated the special case where $P$ is reductive over $\cO_k$. Some technical issues arise in the general setting; these are treated in Sections \ref{sec:Preliminaries} and \ref{sec:Sigma}, especially Sections \ref{sec:filtration_general}, \ref{sec:preparations_coh_Sigmaprime} and \ref{sec:stratification_on_Kr1}.

Our second result is the computation of traces of \textit{unramified very regular} elements of $P(\cO_k)$ acting on $R_{T,U}^\theta$ (Definition \ref{def:veryRegularElements}). The proof is based on the Deligne--Lusztig fixed point formula \cite[Theorem 3.2]{DeligneL_76} and adapts ideas of \cite[Theorem 4.2]{DeligneL_76}.

\begin{thm}[cf.\ Theorem \ref{thm:traces}]\label{thm:2}
For any character $\theta \from T(\cO_k) \to \cool^\times$ and any unramified very regular element $g \in P(\cO_k)$,
\begin{equation*}
\Tr(g, R_{T,U}^\theta) = \sum_{w \in W_{\bx}(T, Z^\circ(g))^\sigma} (\theta \circ {\rm Ad}(w^{-1}))(g),
\end{equation*}
where the sum ranges over the finite set of $\sigma$-invariant elements in the principal homogeneous space $W_{\bx}(T, Z^\circ(g))$ under the Weyl group of the special fiber of $T$ in the special fiber of $P$ (Section \ref{sec:Weyl_Bruhat}).
\end{thm}
\enlargethispage*{\baselineskip}

When $G$ is any inner form of $\GL_n$ over $k$ and $T$ is an unramified maximal elliptic torus, we prove in \cite{CI_ADLV} that the semi-infinite Deligne--Lusztig set of Lusztig \cite{Lusztig_79} is a scheme and its cohomology realizes the compact induction to $G(k)$ of (an extension of) the $P(\cO_k)$-representations $R_{T,U}^\theta$. Already in this setting, it is not enough to study $R_{T,U}^\theta$ for reductive $P$; for example, when $G$ is an anisotropic modulo center inner form of $\GL_n$, the relevant parahoric is an Iwahori subgroup. This can occur even if $G$ is split: if $G = {\rm Sp}_4$, then there is a conjugacy class of maximal elliptic tori in $G$, such that the relevant $P$ is non-reductive, with the reductive quotient of the special fiber being isomorphic to $\SL_2 \times \SL_2$.

As such, we expect this work to be closely related to the problem of geometrically constructing representations of $p$-adic groups in general. More specifically, we expect that if $T$ is elliptic and $\theta \colon T(k) \rar \cool^\times$ is a sufficiently generic character, then the compact induction to $G(k)$ of (an extension of) the $P(\cO_k)$-representation $R_{T,U}^{\theta}$ is related to the supercuspidal representations constructed by Yu \cite{Yu_01}. Both the irreducibility of and the character formula for $R_{T,U}^\theta$ are crucial ingredients to understanding the corresponding $G(k)$-representation within the context of the local Langlands correspondence.





\subsection*{Acknowledgements}
The first author was partially supported by an NSF Postdoctoral Research Fellowship (DMS-1802905) and by the DFG via the Leibniz Prize of Peter Scholze. The second author was supported by the DFG via the Leibniz Preis of Peter Scholze.

\section{Preliminaries}\label{sec:Preliminaries}

\subsection{Notation}\label{sec:gen_not} We denote by $k$ a non-archimedean local field with residue field $\FF_q$ of prime characteristic $p$, and by $\breve{k}$ the completion of a maximal unramified extension of $k$. We denote by $\mathcal{O}_k$, $\mathfrak{p}_k$ (resp.\ $\cO$, $\mathfrak{p}$) the integers and the maximal ideal of $k$ (resp.\ $\breve k$). The residue field of $\breve k$ is an algebraic closure $\overline{\FF}_q$ of $\FF_q$. We write $\sigma$ for the Frobenius automorphism of $\breve k$, which is the unique $k$-automorphism of $\breve k$, lifting the $\FF_q$-automorphism $x \mapsto x^q$ of $\overline{\FF}_q$. Finally, we denote by $\varpi$ a uniformizer of $k$ (and hence of $\breve k$) and by $\ord = \ord_{\breve k}$ the valuation of $\breve k$, normalized such that $\ord(\varpi) = 1$.

If $k$ has positive characteristic, we let $\bW$ denote the ring scheme over $\FF_q$ where for any $\FF_q$-algebra $A$, $\bW(A) = A[\![\varpi]\!]$. If $k$ has mixed characteristic, we let $\bW$ denote the $k$-ramified Witt ring scheme over $\FF_q$ so that $\bW(\FF_q) = \cO_k$ and $\bW(\overline \FF_q) = \cO$. 
As the Witt vectors are only well behaved on perfect $\FF_q$-algebras, algebro-geometric considerations when $k$ has mixed characteristic are taken up to perfection. We fix the following convention. 

\medskip

\noindent \textbf{Convention.} If $k$ has mixed characteristic, whenever we speak of a scheme over its residue field $\FF_q$, we mean a \emph{perfect scheme}, that is a functor a set-valued functor on perfect $\FF_q$-algebras. 

\medskip


For results on perfect schemes we refer to \cite{Zhu_17, BhattS_17}. Note that passing to perfection does not affect the $\ell$-adic \'etale cohomology; thus for purposes of this paper, we could in principle pass to perfection in all cases. However, in the equal characteristic case working on non-perfect rings does not introduce complications, and we prefer to work in this slightly greater generality.

Fix a prime $\ell \neq p$ and an algebraic closure $\overline \QQ_\ell$ of $\QQ_\ell$. The field of coefficients of all representations is assumed to be $\overline \QQ_\ell$ and all cohomology groups throughout are compactly supported $\ell$-adic \'etale cohomology groups. 



\subsection{Group-theoretic data}\label{sec:groupth_data}

We let $G$ be a connected reductive group over $k$, such that the base change $G_{\breve k}$ to $\breve k$ is split. Let $T$ be a $k$-rational, $\breve k$-split maximal torus in $G$. 
Let $\sB_{\breve k}$ and $\sB_k$ denote the Bruhat--Tits building of the adjoint group of $G$ over $\breve k$ and over $k$, and let $\sA_{T, \breve k} \subseteq \sB_{\breve k}$ denote the apartment of $T$. Note that there is a natural action of $\Aut(\breve k/k) = \langle \sigma \rangle$ on $\sB_{\breve k}$ and on $\sA_{T,\breve k}$, and that $\sB_k = \sB_{\breve k}^{\langle\sigma\rangle}$. 

Let $X^{\ast}(T)$ and $X_{\ast}(T)$ denote the group of characters and cocharacters of $T$. We denote by $\langle \cdot , \cdot \rangle \colon X^{\ast}(T) \times X_{\ast}(T) \rightarrow \bZ$ the natural $\bZ$-linear pairing between them. We extend it to the uniquely determined $\bR$-linear pairing $\langle \cdot , \cdot \rangle \colon X^{\ast}(T)_{\bR} \times X_{\ast}(T)_{\bR} \rightarrow \bR$, where we write $M_{\bR} = M \otimes_{\bZ} \bR$ for a $\bZ$-module $M$.

Denote by $\Phi$ the set of roots of $T$ in $G_{\breve k}$ and for a root $\alpha \in \Phi$ let $U_{\alpha} \subseteq G_{\breve k}$ denote the corresponding root subgroup. There is an action of $\langle \sigma \rangle$ on $\Phi$. Fix a Chevalley system $u_{\alpha} \colon \bG_a \stackrel{\sim}{\rightarrow} U_{\alpha}$ for $G_{\breve k}$ (cf. e.g. \cite[4.1.3]{BruhatT_84}). To any root $\alpha \in \Phi$ we can attach the valuation $\varphi_{\alpha} \colon U_{\alpha}(\breve k) \rightarrow \bZ$ given by $\varphi_{\alpha}(u_{\alpha}(y)) = \ord(y)$. The set of valuations $\{\varphi_{\alpha}\}_{\alpha \in \Phi}$ defines a point ${\bf x}_0$ in the apartment $\sA_{T, \breve k}$. Moreover $\sA_{T,\breve k}$ is an affine space under $X_{\ast}(T)_{\bR}$ and the point ${\bf x}_0 + v \in \sA_{T,\breve k}$ for $v \in X_{\ast}(T)_{\bR}$ corresponds to the valuations $\{ \widetilde \varphi_{\alpha} \}_{\alpha \in \Phi}$ of the root datum given by $\widetilde\varphi_{\alpha}(u) = \varphi_{\alpha}(u) + \langle \alpha, v \rangle$ (see \cite[6.2]{BruhatT_72}).

We let $U, U^-$ be the unipotent radicals of two opposite $\breve k$-rational Borel subgroups of $G_{\breve k}$ containing $T$. 

\subsection{Affine roots and filtration on the torus}
We have the set $\Phi_{\rm aff}$ of affine roots of $T$ in $G_{\breve k}$. It is the set of affine functions of $\sA_{T,\breve k}$ defined as 
\[
\Phi_{\rm aff} = \{ {\bf x} \mapsto \alpha({\bf x} - {\bf x}_0) + m \colon \alpha \in \Phi, m \in \bZ \}.
\]
Denote the affine root $(\alpha,m) \from {\bf x} \mapsto \alpha({\bf x} - {\bf x}_0) + m$ and call $\alpha$ its vector part. We have the affine root subgroups $\breve U_{\alpha,m} \subseteq U_{\alpha}(\breve k)$, defined by 
\[
\breve U_{\alpha,m} = \{ u \in U_{\alpha}(\breve k) \colon u = 1 \text{ or } \varphi_{\alpha}(u) \geq m \}
\]
They define a descending separated filtration of $U_{\alpha}(\breve k)$. There is a natural action of Frobenius $\sigma$ on the set of affine roots. We make it explicit:
\begin{lm}\label{lm:FrobeniusOnAffineRoots}
Let $(\alpha,m) \in \Phi_{\rm aff}$. Then $\sigma(\alpha,m) = (\sigma(\alpha), m - \langle \alpha , \sigma({\bf x}_0) - {\bf x}_0 \rangle )$. 
\end{lm}

\begin{proof}
We have $\sigma(\alpha,m) = (\sigma(\alpha),m')$ for some $m' \in \bZ$. The evaluation of the affine-linear form $(\alpha,m)$ on the apartment $\cA_{T,\breve k}$ is $\sigma$-linear, thus we have for all ${\bf x} \in \cA_{T,\breve k}$: 
\[
\sigma(\alpha,m)({\bf x}) = (\alpha,m)(\sigma^{-1}({\bf x})) = \langle \alpha, \sigma^{-1}({\bf x}) - {\bf x}_0 \rangle + m
= \langle \sigma(\alpha), {\bf x} - {\bf x}_0 \rangle + m - \langle \sigma(\alpha), \sigma({\bf x}_0) - {\bf x}_0 \rangle. \]
On the other side, $(\sigma(\alpha),m')({\bf x}) = \langle \sigma(\alpha), {\bf x} - {\bf x}_0 \rangle + m'$, whence the lemma.
\end{proof}

Let $\widetilde{\bR} = \bR \cup \{ r+ \colon r \in \bR\} \cup \{\infty \}$ denote the ordered monoid as in \cite[6.4.1]{BruhatT_72}. Let $\breve T^0 \subseteq T(\breve k)$ be the maximal bounded subgroup. 
For $r \in \widetilde{\bR}_{\geq 0} \sm \{\infty \}$, we have a descending separated filtration of $\breve T^0$ given by
\[
\breve{T}^r = \{t \in \breve T^0 \colon \ord(\chi(t) - 1) \geq r \,\,\forall \chi \in X^{\ast}(T) \}. \]

\subsection{Parahoric subgroups, Moy--Prasad filtration and integral models}\label{sec:parahorics}
Fix a point ${\bf x} \in \sA_{T,\breve k}$. Following Bruhat and Tits \cite[5.2.6]{BruhatT_84}, there is a parahoric group scheme $P_{\bf x}$ over $\cO$ attached to ${\bf x}$, with generic fiber $G$, and with connected special fiber.
The group $\breve P_{\bf x} \colonequals P_{\bf x}(\cO)$ is generated by $\breve T^0$ and $\breve U_{\alpha,m}$ for all $(\alpha,m) \in \Phi_{\rm aff}$ such that $\langle \alpha, {\bf x} - {\bf x_0} \rangle \geq -m$ (that is, $(\alpha,m)({\bf x}) \geq 0$). The schematic closure of $T$ in $P_{{\bf x}}$ is the connected N\'eron model of $T$. We denote it again by $T$. We have $T(\cO) = \breve T^0$. (As $G_{\breve k}$ is split, condition (T) of \cite[8.1]{Yu_02} is satisfied. The claim about the closure of $T$ in $P_{\bf x}$ follows e.g. from \cite[Corollary 8.6(ii)]{Yu_02}. Again, because $G_{\breve k}$ is split, it also follows (\cite[4.6.1]{BruhatT_84}) that the connected N\'eron model of $T$ is equal to the maximal subgroup scheme of finite type of the lft model of $T$. The $\cO$-points of the latter are equal to $\breve T^0$, hence we indeed have $T(\cO) = \breve T^0$.)

The Moy--Prasad filtration on $\breve P_{\bf x}$ is given by the series of normal subgroups $\breve P_{{\bf x}}^r \subseteq \breve P_{{\bf x}}$ $(r \in \widetilde{\bR}_{\geq 0} \sm \{\infty\})$, generated by $\breve{T}^r$ and $\breve{U}_{(\alpha,m)}$ for all $(\alpha,m) \in \Phi_{\rm aff}$ such that $\langle \alpha, {\bf x} - {\bf x}_0 \rangle \geq r-m$. By \cite[8.6 Corollary]{Yu_02}, there is an unique smooth $\cO$-model $P_{{\bf x}}^r$ of $G$, such that $P_{{\bf x}}^r(\cO) = \breve P_{{\bf x}}^r$. Moreover, part (ii) of the same corollary describes the schematic closures of $U_{\alpha}$, $T$ in $P_{{\bf x}}^r$, and in particular, we have
\begin{equation}\label{eq:intersection_of_P_and_root} 
\breve P_{\bf x}^r \cap U_{\alpha}(\breve k) = \breve U_{\alpha,\lceil r - \langle \alpha, {\bf x} - {\bf x}_0 \rangle \rceil} \quad \text{and} \quad \breve P_{\bf x}^r \cap T(\breve k) = \breve T^r.
\end{equation}
Note that for $r\in \bR_{\geq 0}$, we have $\breve P_{{\bf x}}^{r+} = \bigcup_{s \in \bR, s>r} \breve P_{{\bf x}}^s$. For further properties of the Moy--Prasad filtration we refer to \cite[\S2.6]{MoyP_94}, and for further properties of the smooth models $P_{{\bf x}}^r$ we refer to \cite{Yu_02}.

Assume now that ${\bf x} \in \sA_{T,\breve k} \cap \sB_k$. Then all group schemes $P_{\bf x}$, $P_{\bf x}^r$ descend to smooth group schemes over $\cO_k$, again denoted by $P_{\bf x}$, $P_{\bf x}^r$ (cf. \cite[\S9.1]{Yu_02}). In particular, all groups $\breve P_{{\bf x}}^r$ ($r \geq 0$) are $\sigma$-stable (this can also be deduced from Lemma \ref{lm:FrobeniusOnAffineRoots}, which shows that $\sigma$ maps $\breve U_{\alpha,\lceil r - \langle \alpha, {\bf x} - {\bf x}_0 \rangle \rceil}$ isomorphically onto $\breve U_{\sigma(\alpha),\lceil r - \langle \sigma(\alpha), {\bf x} - {\bf x}_0 \rangle \rceil}$), and $P_{\bf x}(\cO_k) = \breve P_{\bf x}^\sigma$ and $P_{\bf x}^r(\cO_k) = (\breve P_{\bf x}^r)^\sigma$. 

\subsection{Moy--Prasad quotients}\label{sec:groups_positive_loops}
For a scheme $\fX$ over $\cO_k$ (resp. over $\cO$), the functor of positive loops $L^+ \fX$ is the functor on $\FF_q$-algebras (resp. $\overline{\FF}_q$-algebras) given by 
\[
L^+\fX(R) = \fX(\bW(R)) 
\]
If $\fX$ is affine and of finite type, then $L^+\fX$ is represented by an affine scheme. 

Let ${\bf x} \in \sA_{T,\breve k} \cap \sB_k$ be as in Section \ref{sec:parahorics}. We have the infinite-dimensional affine $\FF_q$-group scheme $L^+P_{\bf x}$, and will now introduce convenient finite-dimensional quotients of it. Let $r\in \bZ_{\geq 1}$. We consider the fpqc quotient sheaf $\bG_r := L^+P_{\bf x}/L^+P_{\bf x}^{(r-1)+}$. By \cite[Proposition 4.2(ii)]{CI_ADLV} it is representable by a smooth affine group scheme over $\FF_q$ of finite type, which we again denote by $\bG_r$. From \cite[Theorem 8.8]{Yu_02}, along with the fact that $L^+P_{\bf x}^{(r-1)+}$ is pro-unipotent, it follows by taking Galois cohomology,
\[
\breve G_r \colonequals \bG_r(\overline{\FF}_q) = \breve P_{{\bf x}}/\breve{P}_{{\bf x}}^{(r-1)+} \quad \text{and } \quad \breve G_r^\sigma = \bG_r(\FF_q) = (\breve P_{{\bf x}}/\breve P_{{\bf x}}^{(r-1)+})^\sigma.
\]

For $r\geq s\geq 1$ we have natural surjections of $\FF_q$-groups $L^+P_{\bf x} \rar \bG_r \rar \bG_s$.  We write $\bG_r^s = \ker(\bG_r \tar \bG_s)$ and $\breve G_r^s \colonequals \bG_r^s(\overline \FF_q)$. Moreover, we also have natural surjections $\bG_2 \rar P_{\bf x} \otimes_{\cO_k} \FF_q \rar (P_{\bf x} \otimes_{\cO_k} \FF_q)^{\rm red } = \bG_1$ identifying $\bG_1$ with the reduced quotient of the special fiber of $P_{\bf x}$ and of each $\bG_r$. 


\subsection{Subgroups of $\bG_r$}\label{sec:subgroups_of_bGr}
Let $H \subseteq G_{\breve k}$ be a closed subgroup scheme. Let $r \in \bZ_{\geq 1}$. We will attach to $H$ the subgroup $\bH_r \subseteq \bG_{r, \overline{\FF}_q}$ as follows. The schematic closure $H_{\bf x}$ of $H$ in $P_{{\bf x}, \cO}$ is flat (by \cite[1.2.6]{BruhatT_72} as $\cO$-flat is equivalent to $\cO$-torsion free). It follows that $H_{\bf x}$ is a closed subgroup scheme of $P_{{\bf x}, \cO}$ (\cite[1.2.7]{BruhatT_72}). Apply $L_r^+$ to the inclusion $H_{\bf x} \subseteq  P_{\bf x,\cO}$ to obtain the subgroup scheme $L_r^+H_{\bf x} \subseteq L_r^+P_{\bf x,\cO}$. The last inclusion is a closed immersion (e.g. by \cite[Corollary 2 on p. 639]{Greenberg_61}). We define the closed subgroup scheme $\bH_r \subseteq \bG_{r,\overline{\FF}_q}$ as the image of $L_r^+H_{\bf x}$ under $L_r^+P_{\bf x,\cO} \tar \bG_{r,\overline{\FF}_q}$. We write $\bH_r^s \colonequals \ker(\bH_r \rightarrow \bH_s)$ and $\bH_r^{s,\ast} \colonequals \bH_r^s \smallsetminus \bH_r^{s+1}$. 

Suppose now additionally that $H_{\bf x}$ is smooth. Then $L_r^+H_{\bf x}$ is reduced (one could e.g. use \cite[Corollary 2 on p. 264]{Greenberg_63}), and hence $\bH_r$ is too. If $H$ is already defined over the finite subextension of $\breve k/k$ of degree $d$, then $H_{\bf x}$ is defined over the integers of this subextension. This implies that  $\bH_r(\overline{\FF}_q)$ is stable under the action of $\sigma^d$. Hence $\bH_r$ is defined over $\FF_{q^d}$ (here we use that $\bH_r$ is reduced separated scheme of finite type over $\FF_q$). 

Using the procedure described above we obtain the closed $\FF_q$-subgroup $\bT_r \subseteq \bG_r$ attached to $T \subseteq G$. Analogously, we have the subgroups $\bU_r, \bU_r^- \subseteq \bG_{r,\overline{\FF}_q}$ corresponding to $U,U^- \subseteq G_{\breve k}$ and for any root $\alpha \in \Phi$ the subgroup $\bU_{r,\alpha} \subseteq \bG_{r,\overline{\FF}_q}$ corresponding to $U_{\alpha}$. Note that all these are reduced connected closed subgroups of $\bG_{r,\overline{\FF}_q}$. Moreover, $\bU_{r,\alpha}$ is defined over $\FF_{q^d}$ where $d \in \bZ_{\geq 1}$ is the smallest positive integer such that $\sigma^d(\alpha) = \alpha$ in $\Phi$ (indeed the group $U_{\alpha,{\bf x}}$ is smooth by \cite[8.3 Theorem (ii)]{Yu_02}), and a similar statement holds for $\bU_r, \bU_r^-$.

For any reduced $\overline{\FF}_q$-subscheme $\bX_r \subseteq \bG_{r, \overline \FF_q}$, we define $\breve X_r \colonequals \bX_r(\overline \FF_q) \subseteq \bG_r(\overline \FF_q) = \breve G_r$. Thus for example we write $\breve U_{\alpha,r}^a = \bU_{\alpha,r}^a(\overline{\FF}_q)$ for $\alpha \in \Phi$ and $1\leq a \leq r-1$. Following Lusztig, we denote by $\cT$ the groups $\bT_r^{r-1}$. For $\alpha \in \Phi$, let $T^\alpha \subset T_{\breve k} \subset G_{\breve k}$ be the unique $1$-dimensional torus contained in the subgroup of $G_{\breve k}$ generated by $U_\alpha$ and $U_{-\alpha}$; let $\bT_r^\alpha$ be the corresponding subgroup scheme of $\bG_{r,\overline{\FF}_q}$ and write $\cT^\alpha \colonequals \bT_r^{\alpha,r-1}$.

\begin{lm}\label{lm:GGr_generated_by_roots} Let $r \in \bZ_{\geq 1}$ and $1\leq a \leq r-1$. 
\begin{itemize}
\item[(i)] The group $\bG_r$ is generated by $\bT_r$ and all $\bU_{\alpha,r}$ ($\alpha \in \Phi$). 
\item[(ii)] The group $\bG_r^{a+1}$ is generated by $\bT_r^{a+1}$ and all $\bU_{\beta,r}^{a+1}$ ($\alpha \in \Phi$)
\end{itemize}
\end{lm}
\begin{proof}
As $\bG_r$, $\bG_r^a$ are smooth affine $\FF_q$-groups, the assertions can be checked on $\overline{\FF}_q$-points. Now both cases follow from \cite[Theorem 8.3]{Yu_02} applied to the smooth models $P_{\bf x}$ and $P_{\bf x}^{a+}$ of $G$ respectively (note that with notations as in \emph{loc.~cit.}, the group $G(k)_{\bx,f}$ is by definition the one generated by all $U_a(k)_{\bx,f(a)}$).
\end{proof}


\begin{rem} Let $U'$ be the unipotent radical of some other Borel subgroup of $G_{\breve k}$ containing $T$. Although $U$ and $U'$ are conjugate by an element of $G(\breve k)$, the groups $\bU_r(\overline{\FF}_q)$ and $\bU'_r(\overline{\FF}_q)$ need not be isomorphic. For example, let $G$ be the anisotropic modulo center inner form of $\GL_3$ (it splits over $\breve k$ and its $k$-points are isomorphic to the units of a division algebra over $k$). Let ${\bf x}$ be the unique point in $\mathscr{B}_k$. Then $\bG_1 = \bT_1$ is a torus and (after an appropriate choice of ${\bf x}_0$) one has $\bG_2(\overline{\FF}_q) = \left(\begin{smallmatrix} \bW_2(\overline{\FF}_q)^\times & \overline{\FF}_q & \overline{\FF}_q \\ \varpi \overline{\FF}_q & \bW_2(\overline{\FF}_q)^\times & \overline{\FF}_q \\ \varpi \overline{\FF}_q & \varpi \overline{\FF}_q & \bW_2(\overline{\FF}_q)^\times\end{smallmatrix} \right)$ with obvious multiplication.  Now, let $U$ and $U'$ be the group of upper- and lower-triangular unipotent matrices in $G$. Then $\bU_2 = \bU_2^1$ is non-abelian, whereas $\bU_2' = \bU_2^{\prime,1}$ is abelian. \hfill $\Diamond$
\end{rem}

\subsection{The groups $\bU_{\alpha,r}$}

We now give explicit formulas for $\bU_{\alpha,r} \subseteq \bG_r$. 

\begin{Def}\label{def:reductive_nonreductive_roots}
Let ${\bf x} \in \sA_{T,\breve k}$. We call a root $\alpha \in \Phi$ 
\begin{align*} \text{reductive} \qquad & \qquad \text{if $\langle \alpha, {\bf x} - {\bf x}_0 \rangle \in \bZ$} \\ \text{non-reductive} \qquad & \qquad \text{otherwise}.
\end{align*} 
For any $\alpha \in \Phi$, we may uniquely write $\langle \alpha, {\bf x} - {\bf x}_0 \rangle = -m_{\alpha} + \varepsilon_{\alpha}$ with $m_{\alpha} \in \bZ$ and $0 \leq \varepsilon_{\alpha} < 1$. We have $m_\alpha = - \lfloor \langle \alpha, {\bf x} - {\bf x}_0 \rangle  \rfloor$.
\end{Def}
Note that $\alpha \in \Phi$ is reductive if and only if $\bU_{\alpha,1} \neq 1$.

\begin{lm}\label{lm:structure_of_rootgroup_subquotients}
Let ${\bf x} \in \sA_{T,\breve k}$ and let $r \in \bZ_{\geq 1}$. Let $\alpha \in \Phi$. We have 
\[m_\alpha + m_{-\alpha} = 
\begin{cases} 
0 & \text{if $\alpha$ is reductive} \\ 1 & \text{otherwise.} 
\end{cases}\]
Moreover, the natural map $\breve P_{\bf x} \tar \bG_r(\overline{\FF}_q)$ induces 
\[
\bU_{\alpha,r}(\overline{\FF}_q) = 
\begin{cases} 
\breve U_{\alpha,m_\alpha} / \breve U_{\alpha,m_\alpha +  r} & \text{if $\alpha$ reductive,} \\ 
\breve U_{\alpha,m_\alpha} / \breve U_{\alpha,m_\alpha +  r - 1} & \text{otherwise.}
\end{cases}
\]
Thus for $a \in \bZ$, $r \geq a \geq 1$, the same map induces
\[
\bU_{\alpha,r}^a(\overline{\FF}_q) = 
\begin{cases} 
\breve U_{\alpha,m_\alpha + a} / \breve U_{\alpha,m_\alpha +  r} & \text{if $\alpha$ reductive,} \\ 
\breve U_{\alpha,m_\alpha + a - 1} / \breve U_{\alpha,m_\alpha +  r - 1} & \text{otherwise.}
\end{cases}
\]
Finally, we have $\bT_r(\overline{\FF}_q) = \breve T^0/\breve T^r$ and $\bT_r^a = \breve T^a/\breve T^r$.
\end{lm}
\begin{proof}
Noting that $\lceil - s \rceil = - \lfloor s \rfloor$ for $s \in \bR$, the lemma follows immediately from \eqref{eq:intersection_of_P_and_root} and the definitions of $\bU_{\alpha,r}$, $\bU_{\alpha,r}^a$ and $\bG_r$.
\end{proof}

We have the following elementary lemma will be useful later.

\begin{lm}\label{lm:epsilons_and_ms}
Let $\alpha, \beta \in \Phi$ and assume that $p,q \in \bZ_{\geq 1}$, such that $p \alpha + q\beta \in \Phi$. Then $pm_{\alpha} + qm_{\beta} - m_{p \alpha + q\beta} = p \varepsilon_{\alpha} + q\varepsilon_{\beta} - \varepsilon_{p \alpha + q\beta} = \lfloor p\varepsilon_{\alpha} + q\varepsilon_{\beta} \rfloor$. In particular, $pm_{\alpha} + qm_{\beta} - m_{p \alpha + q\beta} \geq 0$.
\end{lm}
\begin{proof}
The first equality is immediate. In particular, $p \varepsilon_{\alpha} + q\varepsilon_{\beta} - \varepsilon_{p \alpha + q\beta}$ is an integer. This, along with the fact that $0 \leq \varepsilon_{p\alpha + q\beta} < 1$ by definition, implies the second equality.
\end{proof}

\subsection{Weyl groups. Bruhat--decomposition}\label{sec:Weyl_Bruhat}
We have the group
\[
W_{\bf x}(T) \colonequals (N_G(T)(\breve k) \cap \breve P_{\bf x}^0)/ \breve T^0 
\]
(cf. \cite[Proposition 8]{HainesR_08}), and it coincides with the Weyl group $W(\bT_1, \bG_1)$ of the torus $\bT_1$ in the special fiber $\bG_1$ of  $P_{\bf x}$ (\cite[Proposition 12]{HainesR_08}). It follows that both natural maps in the composition 
\[
W_{\bf x}(T) \rar N_{\bG_r}(\bT_r)(\overline{\FF}_q)/\bT_r(\overline{\FF}_q) \rar N_{\bG_1}(\bT_1)(\overline{\FF}_q)/\bT_1(\overline{\FF}_q)
\] 
are isomorphisms. Here $N_G(H)$ denotes the scheme-theoretic normalizer of the subgroup $H$ of a group $G$ (note that it might be non-reduced, but we have $N_G(H)(\overline{\FF}_q) = N_G(H)_{\rm red}(\overline{\FF}_q) = \{ g \in G(\overline{\FF}_q) \colon g H g^{-1} = H \}$).
We also note that $W_{\bf x}(T)$ coincides with the subgroup of the Weyl group $W = W(T,G)$ of $T$ in $G$ generated by the vector parts of all affine roots $\psi \in \Phi_{\rm aff}$ satisfying $\psi({\bf x}) = 0$ (cf.\ \cite[1.9, 3.5.1]{Tits_79}). It depends only on the facet of $\sB_{\breve k}$ in which ${\bf x}$ lies, not on ${\bf x}$ itself.

We will need a second $k$-rational, $\breve k$-split maximal torus $T'$ of $G$ whose apartment $\sA_{T',\breve k}$ in $\sB_{\breve k}$ passes through the point ${\bf x}$. 
Let $N_G(T,T') = \{g \in G \colon gTg^{-1} = T'\}$ be the transporter from $T$ to $T'$ and analogously, let $N_{\bG_r}(\bT_r, \bT_r')$ be the transporter from $\bT_r$ to $\bT_r'$. (Again, these need not be reduced, but we are interested in $\overline{\FF}_q$-points only.)
We then have the principal homogeneous space 
\begin{equation*}
W_{{\bf x}}(T, T') \colonequals \breve T^0 \backslash (N_G(T, T')(\breve k) \cap \breve P_{\bx}^0) = \bT_r(\overline{\FF}_q) \backslash N_{\bG_r}(\bT_r, \bT_r')(\overline{\FF}_q).
\end{equation*}
under $W_{\bf x}(T)$. Indeed, this follows as $T$ and $T'$ are conjugate by an element of $P_{\bf x}(\cO)$.

%

Let $r \geq 1$. For each $w \in W_{\bf x}(T,T')$ choose a representative $\dot{w} \in N_{\bG_r}(\bT_r,\bT'_r)(\overline{\FF}_q)$, and denote its image in $\bG_1$ again by $\dot w$. We have the Bruhat decomposition $\bG_1 = \bigsqcup_{w\in W_{\bf x}(T,T')} \bG_{1,w}$ of the reductive quotient, where $\bG_{1,w} = \bU_1\dot{w}\bT_1'\bU_1'$. For $r \geq 1$, define $\bG_{r,w}$ to be the pullback of $\bG_{1,w}$ along the natural projection $\bG_r \twoheadrightarrow \bG_1$. Thus $\bG_r = \bigsqcup_{w \in W_{{\bf x}}(T,T')} \bG_{r,w}$. Let $\bK_r \colonequals \bU_r^- \cap \dot w \bU_r^{\prime -} \dot w^{-1}$ and $\bK_r^1 \colonequals \bK_r \cap \bG_r^1$.  

\begin{lm}\label{lm:Bruhat_cells_decomp}
For $r \geq 1$, we have $\bG_{r,w} = \bU_r \bK_r^1 \dot{w} \bT_r'\bU_r '$.
\end{lm}
\begin{proof} 
We compute 
\begin{align*}
\bG_{r,w} &= \bU_r \dot w \bT_r' \bG_r^1 \bU_r'  = \bU_r \dot w \bT_r' \left((\bG_r^1 \cap \bT_r') (\bG_r^1 \cap \bU_r^{\prime-}) (\bG_r^1 \cap \bU_r')\right) \bU_r' \\
&= \bU_r \dot w \bT_r' (\bG_r^1 \cap \bU_r^{\prime-}) \bU_r' = \bU_r \left(\dot w (\bG_r^1 \cap \bU_r^{\prime-}) \dot w^{-1}\right) \dot{w} \bT_r' \bU_r' \\
&= \bU_r \left(\bU_r^- \cap \dot w (\bG_r^1 \cap \bU_r^{\prime -}) \dot w^{-1}\right) \dot w \bT_r' \bU_r' = \bU_r \bK_r^1 \dot w \bT_r' \bU_r', 
\end{align*}
where the second equality follows from \cite[6.4.48]{BruhatT_72}. 
\end{proof}

%

\subsection{Commutation relations}

For two subgroups $H_1,H_2$ of an abstract group $H$, we denote by $[H_1,H_2]$ their commutator. For $x,y \in H$, we write $[x,y] \colonequals x^{-1}y^{-1}xy$. 

For $\alpha \in \Phi$, let $T^{\alpha} \subseteq T$ denote the image of the coroot corresponding to $\alpha$. It is a one-dimensional subtorus. We also write $\breve T^{\alpha,r} = T^{\alpha}(\breve k) \cap \breve T^r$.

\begin{lm}\label{lm:usual_commutator_relations}
\begin{itemize}
\item[(i)] Let $\alpha \in \Phi$ and $r,m \in \widetilde{\bR}$. Then $[\breve T^r, \breve U_{\alpha,m}] \subseteq \breve U_{\alpha, m+r}$. 
\item[(ii)] If $\alpha, \beta \in \Phi$, $\alpha \neq -\beta$, and $m_1,m_2 \in \bZ$, then $[\breve U_{\alpha,m_1}, \breve U_{\beta,m_2}]$ is contained in the group generated by $\breve U_{p\alpha + q\beta, pm_1+qm_2}$ for all $p,q \in \bZ_{\geq 1}$, such that $p\alpha + q\beta  \in \Phi$.
\item[(iii)] Let $\alpha \in \Phi$ and $m_1,m_2 \in \bZ$. Then $[\breve U_{\alpha,m_1}, \breve U_{-\alpha,m_2}] \subseteq \breve T^{\alpha,m_1+m_2}$. For any element $x \in \breve U_{-\alpha,m_2} \sm \breve U_{-\alpha,m_2+1}$, the map $\xi \mapsto [\xi,x]$ induces an isomorphism (of abelian groups)
\[
\lambda_x \colon \breve U_{\alpha,m_1}/\breve U_{\alpha,m_1 + 1} \stackrel{\sim}{\rar} \breve T^{\alpha,m_1+m_2}/\breve T^{\alpha, m_1 + m_2 + 1}.
\]
\end{itemize}
\end{lm}
\begin{proof}
(ii) follows from \cite[(6.2.1)]{BruhatT_72}. (i), (iii): By considering a morphism from $\SL_2$ to $G_{\breve k}$, whose image is generated by $U_{\pm \alpha}$ (as in \cite[(6.2.3) b)]{BruhatT_72}), and pulling back the valuation of the root datum along this morphism, it suffices to prove the same statement for $\SL_2(\breve k)$. This is an immediate computation.
\end{proof}

For two smooth (connected) closed subgroups $\bH_1$, $\bH_2$ of a connected linear algebraic group $\bG$ over a field, we denote by $[\bH_1,\bH_2]$ their commutator ``in the sense of group varieties'' as in \cite[\S2.3]{Borel_91} (it would be more precise to consider the scheme-theoretic commutator, but for our purposes this suffices).


\begin{lemma}\label{lm:comm}
Let $r \geq 2$ and $1 \leq a \leq r-1$. Let $\alpha \in \Phi$.
\begin{enumerate}[label=(\alph*)]
\item
If $\alpha$ is non-reductive, then $[\bG_r^{a+1}, \bU_{\alpha,r}^{r-a}] = 1$.

\item
If $\alpha$ is reductive, then $[\bG_r^a, \bU_{r,\alpha}^{r-a}] = 1$.
\end{enumerate}
\end{lemma}

\begin{proof} It suffices to prove the claims on $\overline{\FF}_q$-points. (a): By Lemma \ref{lm:GGr_generated_by_roots} it suffices to show to show that $[\breve T_r^{a+1}, \breve U_{\alpha,r}^{r-a}] = 1$ and that $[\breve U_{\beta,r}^{a+1}, \breve U_{\alpha,r}^{r-a}] = 1$ ($\forall \beta \in \Phi$) in $\bG_r$. By Lemma \ref{lm:structure_of_rootgroup_subquotients} $\breve T_r^{a+1}$ is the image in $\breve G_r$ of $\breve T^{a+1}$, $\breve U_{\alpha,r}^{r-a}$ is the image of $\breve U_{\alpha, m_{\alpha} + r - a - 1}$, and similar claims hold for all $\beta \in \Phi$. But $[\breve T^{a+1}, \breve U_{\alpha,  m_{\alpha} + r - a - 1}] \subseteq \breve U_{\alpha, r + m_{\alpha}}$ by Lemma \ref{lm:usual_commutator_relations}(i), and $\breve U_{\alpha, r + m_{\alpha}}$ maps to $1$ in $\bG_r$, so $[\breve T_r^{a+1}, \breve U_{\alpha,r}^{r-a}] = 1$ follows. Now assume that $\beta = -\alpha$. Then $-\alpha$ is non-reductive as $\alpha$ is, and by Lemma \ref{lm:usual_commutator_relations}(iii), $[\breve U_{-\alpha,m_{-\alpha} + a}, \breve U_{\alpha, m_{\alpha} + r - a - 1}] \subseteq \breve T^{\alpha, r + m_{\alpha} + m_{-\alpha} - 1} = \breve T^{\alpha, r}$ maps to $1$ in $\bG_r$. This shows $[\breve U_{-\alpha,r}^{a+1}, \breve U_{\alpha,r}^{r-a}] = 1$. Thus we can assume $\beta \in \Phi$, $\beta \neq -\alpha$. We have two cases. 

\medskip
\noindent\emph{Case: $\beta$ is reductive.} Then by Lemma \ref{lm:structure_of_rootgroup_subquotients}, $\breve U_{\beta,r}^{a+1}$ is the image in $\breve G_r$ of $\breve U_{\beta,m_{\beta} +  a + 1}$ and by Lemma \ref{lm:usual_commutator_relations}(ii) we have
\[
[\breve U_{\beta,m_{\beta} +  a + 1}, \breve U_{\alpha, m_{\alpha} + r - a - 1}] \subseteq \prod_{\substack{p,q \in \bZ_{\geq 1} \\ p\alpha + q\beta \in \Phi}} \breve U_{p\alpha + q\beta, p(m_{\alpha} + r - a - 1) + q(m_{\beta} + a + 1)}.
\]
To ensure that this product maps to $1$ in $\breve G_r$, it suffices to show that for all $p,q \in \bZ_{\geq 1}$ with $p\alpha + q\beta \in \Phi$, one has
$p(m_{\alpha} + r - a - 1) + q(m_{\beta} + a + 1) \geq m_{p\alpha + q\beta} + r$, or equivalently, 
\[ pm_{\alpha} + qm_{\beta} - m_{p\alpha + q\beta} + (p-1)(r - a - 1) + (q-1)(a + 1) \geq 0. \]
But this follows from Lemma \ref{lm:epsilons_and_ms}.

\medskip
\noindent \emph{Case: $\beta$ is non-reductive.} By Lemma \ref{lm:structure_of_rootgroup_subquotients}, $\breve U_{\beta,r}^{a+1}$ is the image in $\breve G_r$ of $\breve U_{\beta,m_{\beta} +  a}$ and by Lemma \ref{lm:usual_commutator_relations}(ii) we have
\[
[\breve U_{\beta,m_{\beta} +  a}, \breve U_{\alpha, m_{\alpha} + r - a - 1}] \subseteq \prod_{\substack{p,q \in \bZ_{\geq 1} \\ p\alpha + q\beta \in \Phi}} \breve U_{p\alpha + q\beta, p(m_{\alpha} + r - a - 1) + q(m_{\beta} + a)}.
\]
To show that the image of this product vanishes in $\bG_r$, we have to show that each single term does. Assume that $p\alpha + q\beta$ occurs in the product and is non-reductive. Then vanishing of $\breve U_{p\alpha + q\beta, p(m_{\alpha} + r - a - 1) + q(m_{\beta} + a)}$ in $\breve G_r$ amounts to the inequality
\[ pm_{\alpha} + qm_{\beta} - m_{p\alpha + q\beta} + (p-1)(r - a - 1) + (q-1)a \geq 0, \]
which holds true by Lemma \ref{lm:epsilons_and_ms}. Assume finally that $p\alpha + q\beta$ occurs in the product and is reductive. Then vanishing of $\breve U_{p\alpha + q\beta, p(m_{\alpha} + r - a - 1) + q(m_{\beta} + a)}$ in $\breve G_r$ amounts to the inequality
\[
pm_{\alpha} + qm_{\beta} - m_{p\alpha + q\beta} + (p-1)(r - a - 1) + (q-1)a \geq 1,
\]
or equivalently,
\[
\lfloor p\varepsilon_{\alpha} + q\varepsilon_{\beta} \rfloor + (p-1)(r - a - 1) + (q-1)a \geq 1,
\]
i.e. it suffices to show that $p\varepsilon_{\alpha} + q\varepsilon_{\beta} \geq 1$. But as $p\alpha + q\beta$ is reductive,
\begin{equation}\label{eq:residue_is_at_least_one}
\bZ \ni \langle p\alpha + q\beta, {\bf x} - {\bf x}_0 \rangle = p\langle \alpha, {\bf x} - {\bf x}_0 \rangle + q\langle \beta, {\bf x} - {\bf x}_0 \rangle = -pm_{\alpha} - qm_{\beta} + p \varepsilon_{\alpha} + q \varepsilon_{\beta}.
\end{equation}
As $-pm_{\alpha} - qm_{\beta} \in \bZ$, we deduce $p \varepsilon_{\alpha} + q \varepsilon_{\beta} \in \bZ$. On the other side $\varepsilon_{\alpha}, \varepsilon_{\beta} > 0$ (as $\alpha,\beta$ non-reductive), and hence $p \varepsilon_{\alpha} + q \varepsilon_{\beta} > 0$. Thus, $p \varepsilon_{\alpha} + q \varepsilon_{\beta} \geq 1$. This finishes the proof of (a).

\mbox{}

\noindent (b): We have $[\breve T^a, \breve U_{\alpha,m_{\alpha} + r - a}] \subseteq \breve U_{\alpha,m_{\alpha} + r}$ by Lemma \ref{lm:usual_commutator_relations}(i), and the latter group maps to $1$ in $\breve G_r$. Thus $[\breve T_r^a, \breve U_{\alpha,r}^{r-a}] = 1$. Further, Lemma \ref{lm:usual_commutator_relations}(iii) shows
\[ 
[\breve U_{-\alpha,m_{-\alpha} + a},\breve U_{\alpha,m_{\alpha} + r - a}] \subseteq \breve T^{\alpha, m_{\alpha} + m_{-\alpha} + r} = \breve T^{\alpha, r},
\]
which maps to $1$ in $\breve G_r$. Thus $[\breve U_{-\alpha,r}^a, \breve U_{\alpha,r}^{r-a}] = 1$. Finally, let $\beta \in \Phi$, $\beta \not= -\alpha$. Again we have two cases.

\medskip
\noindent\emph{Case: $\beta$ is reductive.} By Lemma \ref{lm:usual_commutator_relations}(ii),
\[
[\breve U_{\beta,m_{\beta} +  a}, \breve U_{\alpha, m_{\alpha} + r - a}] \subseteq \prod_{\substack{p,q \in \bZ_{\geq 1} \\ p\alpha + q\beta \in \Phi}} \breve U_{p\alpha + q\beta, p(m_{\alpha} + r - a) + q(m_{\beta} + a)},
\]
Now, by Lemma \ref{lm:epsilons_and_ms} we have
\[ p(m_{\alpha} + r-a) + q(m_{\beta} + a) \geq m_{p\alpha + q\beta} + r. \]
So, regardless of whether $p\alpha + q\beta$ is reductive or not, it follows that $\breve U_{p\alpha + q\beta, p(m_{\alpha} + r - a) + q(m_{\beta} + a)}$ maps to $1$ in $\breve G_r$, and hence $[\breve U_{\beta,r}^a, \breve U_{\alpha,r}^{r-a}] = 1$. 

\medskip
\noindent\emph{Case: $\beta$ is non-reductive.} By Lemma \ref{lm:usual_commutator_relations}(ii),
\[
[\breve U_{\beta,m_{\beta} +  a - 1}, \breve U_{\alpha, m_{\alpha} + r - a}] \subseteq \prod_{\substack{p,q \in \bZ_{\geq 1} \\ p\alpha + q\beta \in \Phi}} \breve U_{p\alpha + q\beta, p(m_{\alpha} + r - a) + q(m_{\beta} + a - 1)},
\]
and the proof can be finished exactly as in the ``\emph{$\beta$ non-reductive}''-case of part (a). \qedhere
\end{proof}

\subsection{Regularity of characters} 
Recall the notation $\cT$ from Section \ref{sec:subgroups_of_bGr}. Consider the norm map $N_{\sigma}^{\sigma^m} \from \cT(\overline \FF_q)^{\sigma^m} \to \cT(\overline \FF_q)^{\sigma} = \cT(\FF_q)$ given by $t \mapsto t \sigma(t) \cdots \sigma^{m-1}(t)$. Let $r \in \bZ_{\geq 1}$ be fixed. Following Lusztig \cite[1.5]{Lusztig_04}, we say a character $\chi \colon \cT(\FF_q) \rightarrow \overline{\bQ}_{\ell}^{\times}$ is \emph{regular} if for any $\alpha \in \Phi$ and any $m \geq 1$ such that $\sigma^m(\alpha) = \alpha$, the restriction of $\chi \circ N_{\sigma}^{\sigma^m}$ to $\cT^{\alpha}(\overline \FF_q)^{\sigma^m}$ is non-trivial. A character $\chi$ of $\breve T_r^{\sigma}$ is called \emph{regular} if its restriction $\chi|_{\mathcal{T}(\FF_q)}$ is regular. 

Let $\theta \colon T(k) \rightarrow \overline{\bQ}_{\ell}^{\times}$ be a character of level $r-1$; that is, $\theta$ is trivial on $\breve{T}^{(r-1)+} \cap T(k)$ but nontrivial on $\breve T^{(r-2)+}$. Its restriction to $\breve T^0 \cap T(k)$ can be viewed as a character $\chi$ of $\breve T_r^\sigma = (\breve T^0/\breve T^{(r-1)+})^{\sigma}$. We say $\theta$ is \emph{regular} if $\chi$ is.

\begin{rem}
When $G$ is an inner form of $\GL_n(K)$ and $T$ is a maximal nonsplit unramified torus, then $T(k) \cong L^\times$, where $L$ is the degree-$n$ unramified extension of $k$. If $\theta \from L^\times \to \cool^\times$ is a smooth character trivial on $(\breve T^r)^{\sigma} = U_L^r = 1 + \varpi^r \cO_L$, then $\theta$ being \textit{regular} is the same as being \textit{primitive} in the sense of Boyarchenko--Weinstein \cite[Section 7.1]{BoyarchenkoW_16}. This is closely related to $\theta$ being \textit{minimal admissible} in the sense of Bushnell--Henniart \cite[Section 1.1]{BushnellH_05}. We refer to \cite[Remark 12.1]{CI_ADLV} for a more precise comparison. \hfill $\Diamond$
\end{rem}

\section{The scheme $\Sigma$}\label{sec:Sigma}

We use notation from Section \ref{sec:Preliminaries}. We fix a point ${\bf x} \in \sB_k$, an integer $r \geq 1$, and two maximal tori $T,T'$ of $G$ defined over $k$, split over $\breve k$, and such that ${\bf x} \in \sA_{T,\breve k} \cap \sA_{T,\breve k}$. Further, we fix choose pairs of unipotent radicals of opposite Borels $U,U^-$ (attached to $T$) and $U',U^{\prime,-}$ (attached to $T'$) in $G_{\breve k}$. The construction from Section \ref{sec:subgroups_of_bGr}, this gives the $\FF_q$-groups $\bG_r, \bT_r, \bU_r, \bU_r^-, \bT_r', \bU_r', \bU_r^{\prime,-}$. 

\subsection{Definition of $\Sigma$, $\Sigma_w$} \label{sec:def_Sigma_Sigma_w}

Attached to $(T, U), (T', U')$, we consider the following locally closed reduced subscheme of $\sigma(\bU_r) \times \sigma(\bU'_r) \times \bG_r$ whose $\overline \FF_q$-points are given by
\begin{equation*}
\Sigma(\overline \FF_q) \colonequals \left\{(x,x',y) \in \sigma(\breve U_r) \times \sigma(\breve U'_r) \times \breve G_r \colon x\sigma(y) = y x' \right\}.
\end{equation*}
The scheme $\Sigma$ decomposes into a disjoint union of locally closed subsets, $\Sigma = \coprod_{w \in W_{\bf x}(T, T')} \Sigma_w$, where $\Sigma_w$ is the reduced subscheme of $\Sigma$, whose $\overline{\FF}_q$-points are given by 
\begin{equation*}
\Sigma_w(\overline{\FF}_q) \colonequals \left\{ (x,x',y) \in \Sigma(\overline{\FF}_q) \colon y \in \bG_{r,w}(\overline{\FF}_q) \right\},
\end{equation*}
where $\bG_{r,w}$ is as in Section \ref{sec:Weyl_Bruhat}. The group $\breve T_r^{\sigma} \times \breve T_r^{\prime\sigma}$ acts on $\Sigma$ and each $\Sigma_w$ by 
\[
(t,t') \colon (x,x',y) \mapsto (txt^{-1}, t'x't^{\prime -1}, tyt^{\prime -1}). 
\]
The following lemma is completely analogous to \cite[Lemma 1.4]{Lusztig_04}.

\begin{lm}\label{lm:intertwiner_of_characters_nonzero_coh}
Let $r \geq 2$ and let $\theta \colon \breve T_r^{\sigma} \rar \cool^{\times}$, $\theta' \colon \breve T_r^{\prime \sigma} \rar \cool^{\times}$ be characters such that $H_c^j(\Sigma)_{\theta^{-1},\theta'} \neq 0$ for some $j \in \bZ$. Then there exist $n \geq 1$ and $g \in N_{\bG_r}(\bT'_r, \bT_r)^{\sigma^n}$ such that ${\rm Ad}(g)$ carries $\theta|_{\cT^{\sigma}} \circ N_{\sigma}^{\sigma^n}$ to $\theta'|_{\cT^{\prime \sigma}} \circ N_{\sigma}^{\sigma^n}$.
\end{lm}

\begin{proof}
The proof of \cite{Lusztig_04} applies. The only point where one must be careful is the claim that $\cT$ and $\cT'$ centralize $\bG^1_r$ (this is used to extend the action of $\cT(\FF_q) \times \cT'(\FF_q)$ on a covering of $\Sigma_w$ to an action of a connected group). Passing to $\overline \FF_q$-points, this is the claim that the subgroups $\breve T^{(r-2)+}/\breve T^{(r-1)+} = \breve T^{(r-1)}/\breve T^{(r-1)+}$ and $\breve T^{\prime (r-2)+}/\breve T^{\prime (r-1)+} = \breve T^{\prime (r-1)}/\breve T^{\prime (r-1)+}$ centralize $\breve P_{\bx}^{0+}/\breve P_{\bx}^{(r-1)+}$. By general properties of the Moy--Prasad filtration, $[P_{\bx}^{0+}, P_{\bx}^{(r-1)}] \subseteq P_{\bx}^{(r-1)+}$, which verifies the claim.
%
\end{proof}

\subsection{Alternating sum of cohomology of $\Sigma$}

Our main result about $\Sigma$ is the following generalization of \cite[Lemma 1.9]{Lusztig_04}.

\begin{theorem}\label{thm:altsum_coh_Sigma}
Let $\theta$ and $\theta'$ be characters of $\breve T_r^{\sigma}$ and $\breve T_r^{\prime \sigma}$ respectively, and assume that $\theta$ is regular. Then 
\[
\sum_{i \in \bZ} \dim H_c^i(\Sigma, \cool)_{\theta^{-1},\theta'} = \#\{ w \in W_{\bf x}(T,T')^{\sigma} \colon \theta \circ {\rm Ad}(\dot w) = \theta' \}.
\]
\end{theorem}
\begin{proof}
Since $\Sigma = \bigsqcup_w \Sigma_w$, it is enough to show that $\sum_{i\in \bZ} (-1)^i \dim H_c^i(\Sigma_w, \overline{\mathbb{Q}}_{\ell})_{\theta^{-1},\theta'}$ is $1$ if $w \in W_{\bf x}(T,T')^\sigma$ and $\theta \circ Ad(\dot w) = \theta'$, and is $0$ otherwise. Fix a $w \in W_{\bf x}(T,T')$. Let $\widehat \Sigma_w$ be the locally closed reduced subscheme of $\sigma(\bU_r) \times \sigma(\bU_r') \times \bU_r \times \bU_r' \times \bK_r^1 \times \bT_r'$ such that
\begin{align*}
\widehat \Sigma_w(\overline \FF_q) = \{ (x,x',u,u',z,\tau') \in \sigma(\breve U_r) \times \sigma(\breve U_r') \times \breve U_r \times &\breve U_r' \times \breve K_r^1 \times \breve T_r' \colon \\
&x\sigma(uz\dot{w} \tau' u') = uz\dot{w} \tau' u' x' \},
\end{align*}
and define an action of $\breve T_r^\sigma \times \breve T_r^{\prime \sigma}$ by:
\begin{equation}\label{eq:torus_action_on_widehatSigma}
(t,t') \colon (x,x',u,u',z,\tau') \mapsto (txt^{-1}, t'x't^{\prime -1}, tut^{-1}, t'u't^{\prime -1}, tzt^{-1}, \dot{w}^{-1}t\dot{w}\tau' t^{\prime -1}).
\end{equation}
As the projection $\widehat \Sigma_w \rightarrow \Sigma_w$ is Zariski-locally trivial fibration (as being Zariski-locally trivial is preserved under base change, and $\Sigma_w = \Sigma \times_{\bG_r} \bG_{r,w}$, this follows from Lemma \ref{lm:Bruhat_cells_decomp}), the alternating sum of the cohomology does not change if we pass from $\Sigma_w$ to $\widehat \Sigma_w$. Thus to finish the proof of the theorem it is enough to show that
\begin{equation}\label{eq:claim_for_coh_of_widehatSigma}
\sum_{i\in \bZ} (-1)^i \dim H_c^i(\widehat \Sigma_w, \overline{\mathbb{Q}}_{\ell})_{\theta^{-1},\theta'} = \begin{cases} 1 & \text{if $w \in W_{\bf x}(T,T')^\sigma$ and $\theta \circ Ad(\dot w) = \theta'$,} \\ 0 & \text{otherwise.} \end{cases}
\end{equation}
We make the change of variables replacing $x\sigma(u)$ by $x$ and $x'\sigma(u')^{-1}$ by $x'$, and rewrite $\widehat \Sigma_w(\overline \FF_q)$ as
\begin{equation*}
\widehat \Sigma_w(\overline \FF_q) = \{ (x,x',u,u',z,\tau') \in \sigma(\breve U_r) \times \sigma(\breve U_r') \times \breve U_r \times \breve U_r' \times \breve K_r^1 \times \breve T_r' \colon x\sigma(z\dot{w} \tau') = uz\dot{w} \tau' u' x' \},
\end{equation*}
and the torus action is still given by \eqref{eq:torus_action_on_widehatSigma}.
Define a partition into locally closed subsets $\widehat \Sigma_w = \widehat \Sigma_w' \sqcup \widehat \Sigma_w^{\prime\prime}$ by
\begin{align*}
\widehat \Sigma_w'(\overline{\FF}_q) &= \{ (x,x',u,u',z,\tau') \in \widehat \Sigma_w(\overline{\FF}_q) \colon z \neq 1 \}, \\
\widehat \Sigma_w^{\prime\prime}(\overline{\FF}_q) &= \{ (x,x',u,u',z,\tau') \in \widehat \Sigma_w(\overline{\FF}_q) \colon z = 1 \}. 
\end{align*}
Both subsets are stable under the $\breve T_r^\sigma \times \breve T_r^{\prime \sigma}$-action. Now, 
\begin{equation*}
\sum_{i \in \bZ} (-1)^i \dim H_c^i(\widehat \Sigma_w'', \overline{\mathbb{Q}}_{\ell})_{\theta^{-1}, \theta'} = \begin{cases}
1 & \text{if $w \in W_{\bx}(T, T')^\sigma$ and $\theta \circ Ad(\dot w) = \theta'$,} \\
0 & \text{otherwise,}
\end{cases}
\end{equation*}
has literally the same proof as the claim (b) in the proof of \cite[Lemma 1.9]{Lusztig_04}. Further, we show in Section \ref{sec:hatSigmaprime} (after some preparations in Sections \ref{sec:filtration_general}-\ref{sec:stratification_on_Kr1}), under the assumption that $\theta$ is regular, that
\begin{equation}\label{e:prime}
\sum_{i \in \bZ} (-1)^i \dim H_c^i(\widehat \Sigma_w', \overline{\mathbb{Q}}_{\ell})_{\theta^{-1}, \theta'} = 0,
\end{equation}
so \eqref{eq:claim_for_coh_of_widehatSigma} holds.
\end{proof}

\subsection{Filtration of $\bG_{a+1}^a$}\label{sec:filtration_general}

The main difference of the present article to \cite{Lusztig_04}, is that $\bG_2^1$ is not abelian if (and only if) $P_{\bf x}$ is not reductive, i.e., if ${\bf x}$ is not a hyperspecial point. To deal with this problem, we need a refinement of the filtration of $\bG_r^1$ by its subgroups $\bG_r^a$ for $1 \leq a \leq r-1$. For $a\geq 1$, we define a filtration of $\bG^a_{a+1}$ as follows: let 
\begin{equation*}
H(1) \colonequals \text{subgroup of $\bG^a_{a+1}$ generated $\bT_{a+1}^a$ and $\bU_{\alpha,a+1}^a$ for all reductive $\alpha \in \Phi$},
\end{equation*}
and for all $0 \leq \varepsilon < 1$, let
\begin{equation*}
H(\varepsilon) \colonequals \text{subgroup of $\bG^a_{a+1}$ generated by $H(1)$ and all $\bU_{\alpha,a+1}^a$ for $\alpha \in \Phi$, satisfying $\varepsilon_{\alpha} \geq \varepsilon$}.
\end{equation*}
Note that $\bT_{a+1}^a \subseteq H(1) \subseteq H(\varepsilon') \subseteq H(\varepsilon) \subseteq \bG_a^{a+1}$ for all $1 > \varepsilon' \geq \varepsilon > 0$. Moreover, there are only finitely many values of $\varepsilon$ (``jumps'') satisfying $H(\varepsilon) \supsetneq \bigcup_{\varepsilon' > \varepsilon} H(\varepsilon')$. We denote these jumps by $1 \equalscolon \varepsilon_{s+1} > \varepsilon_s > \dots > \varepsilon_1 > 0$ for some $s \geq 0$ (thus $1$ is a  jump by definition). The jumps are independent of $a$. We have $H(\varepsilon_1) = \bG_{a+1}^a$. For $a \leq r-1$, let $p \colon \bG_r^a \tar \bG_{a+1}^a$ be the natural projection, and for $s+1 \geq i \geq 1$, put 
\[
\bG_r^{a,i} \colonequals p^{-1}(H(\varepsilon_i)).
\]
For convenience, we put $\bG_r^{a,s+2} \colonequals \bG_r^{a+1}$. This defines a refinement $\{\bG_r^{a,i}\}_{\substack{r-1\geq a \geq 1\\ s+2\geq i\geq 1}}$ of the filtration $\{\bG_r^a\}_{r-1\geq a\geq 1}$ of $\bG_r^1$, decreasing with respect to the lexicographical ordering on pairs $(a,i)$. For $s+1 \geq i \geq 1$, let $\Phi_i$ be the set of roots ``appearing'' in $H(\varepsilon_i)/H(\varepsilon_{i+1})$:
\[
\Phi_i \colonequals 
\begin{cases}
\{\alpha \in \Phi : \varepsilon_\alpha = 0\} & \text{if $i = s+1$,} \\
\{\alpha \in \Phi : \varepsilon_\alpha = \varepsilon_i\} & \text{if $s \geq i \geq 1$}.
\end{cases}
\]

\begin{lm}\label{lm:g_filtration_abelian}
Let $r \geq 2$ and $r-1 \geq  a \geq 1$.
\begin{itemize}
\item[(i)] Let $a \geq 2$. Then $\bG_r^a/\bG_r^{a+1} = \bG_{a+1}^a$ is abelian, and in particular, for $s+1 \geq i \geq 1$, $\bG_r^{a,i}/\bG_r^{a,i+1}$ is abelian.
\item[(ii)] Let $a = 1$ and $s+1 \geq i \geq 1$. Then $\bG_r^{1,i}$ is normal in $\bG_r^1$ and the quotient $\bG_r^{1,i}/\bG_r^{1,i+1}$ is abelian.
\end{itemize}
\end{lm}
\begin{proof}
It suffices to prove the assertions on $\overline \FF_q$-points. To show (i), notice that if $a \geq 2$, then $[\breve P_{\bf x}^{(a-1)+}, \breve P_{\bf x}^{(a-1)+}] \subseteq \breve P_{\bf x}^{2(a-1)+} \subseteq \breve P_{\bf x}^{a+}$, so it follows that $\breve G^a_{a+1} = \breve P_{\bf x}^{(a-1)+}/\breve P_{\bf x}^{a+}$ is abelian. To establish (ii), it is enough to show that (with $a = 1$)  for any $s + 1 \geq i \geq 1$, $H(\varepsilon_i)$ is normal in $\bG_2^1$ and that $H(\varepsilon_i)/H(\varepsilon_{i+1})$ is abelian. We spend the rest of the proof establishing these two claims. 
Recall that for $s+1 \geq i \geq 1$, $H(\varepsilon_i)$ is generated by $\bT_2^1$ and all $\bU_{\alpha,2}^{1}$ with $\alpha \in \bigsqcup_{j=i}^{s+1} \Phi_j$.

We start with $i = s+1$, i.e.\ the case $H(\varepsilon_{s+1}) = H(1)$. By Lemma \ref{lm:usual_commutator_relations}, $[\breve T_2^1,\breve U_{\alpha,2}^1] = 1$. Let $\alpha \in \Phi_{s+1}$ (thus $\alpha$ is reductive) and let $\beta \in \Phi$ be any non-reductive root. Then $[\breve U_{\alpha,2}^1, \breve U_{\beta,2}^1]$ is the image in $\breve G_2^1$ of 
\begin{equation}\label{eq:some_commut_comp_proof_abelian_quots}
[\breve U_{\alpha,m_{\alpha} + 1}, \breve U_{\beta, m_{\beta}}] \subseteq \prod_{\substack{p,q \in \bZ_{\geq 1} \\ p\alpha + q\beta \in \Phi}} \breve U_{p\alpha + q\beta, p(m_{\alpha} + 1) + qm_{\beta}}.
\end{equation}
Using Lemma \ref{lm:epsilons_and_ms} along with $p \geq 1$, we see that $p(m_{\alpha} + 1) + qm_{\beta} \geq m_{p\alpha + q\beta} + 1$. Thus the contribution of $p\alpha + q\beta$ to the commutator lies in $\breve U_{p\alpha + q\beta, m_{p\alpha + q\beta} + 1}$. From this we deduce $[\bU_{\alpha,2}^1, \bU_{\beta,2}^1] \subseteq H(1)$. Thus if $x \in \bU_{\beta,2}^1$ for any $\beta \in \Phi$, and $y \in \bU_{\alpha,2}^1$, then $xyx^{-1} = [x^{-1}, y^{-1}] y \in H(1)$, which shows that $H(1)$ is normal in $\bG_2^1$. A computation analogous to \eqref{eq:some_commut_comp_proof_abelian_quots} for $\alpha, \beta \in \Phi^+$ both reductive, shows immediately that $[\bU_{\alpha,2}^1, \bU_{\beta,2}^1] = 1$ and $[\bT_2^1,\bU_{\alpha,2}^1] = 1$, so $H(1)$ is abelian.

Next, pick some $s \geq i \geq 1$. We show that $H(\varepsilon_i)$ is normal in $\bG_2^1$. Since we have already established that $H(\varepsilon_{s+1})$ is normal in $\bG_2^1$, it suffices to check as above that for all (non-reductive) $\alpha \in \Phi$ with $\varepsilon_\alpha \geq \varepsilon_i$ and all non-reductive $\beta \in \Phi$, we have $[\bU_{\alpha,2}^1, \bU_{\beta,2}^1] \subseteq H(\varepsilon_i)$. Now, $[\breve U_{\alpha,2}^1, \breve U_{\beta,2}^1]$ is the image in $\breve G_2^1$ of 
\[
[\breve U_{\alpha, m_{\alpha}}, \breve U_{\beta, m_{\beta}}] \subseteq \prod_{\substack{p,q \in \bZ_{\geq 1} \\ p\alpha + q\beta \in \Phi}} \breve U_{p\alpha + q\beta, pm_{\alpha} + qm_{\beta}}.
\]
Now, if $\varepsilon_{p\alpha + q\beta} \geq \varepsilon_i$, then the contribution of $p\alpha + q\beta$ to the commutator is contained in $\bU_{p\alpha + q\beta,2}^1 \subseteq H(\varepsilon_i)$. If $p\alpha + q\beta$ is reductive, the same computation as in \eqref{eq:residue_is_at_least_one} shows that $\breve U_{p\alpha + q\beta, pm_{\alpha} + qm_{\beta}} \subseteq \breve U_{p\alpha + q\beta, m_{p\alpha + q\beta}} \tar \breve U_{p\alpha + q\beta,2}^1$. It remains to handle the case that $p\alpha + q\beta$ is non-reductive with $\varepsilon_{p\alpha + q\beta} < \varepsilon_i$. If $p\varepsilon_\alpha + q\varepsilon_\beta < 1$, then by Lemma \ref{lm:epsilons_and_ms}, $p\varepsilon_\alpha + q\varepsilon_\beta - \varepsilon_{p\alpha + q\beta} = \lfloor p\varepsilon_\alpha + q\varepsilon_\beta \rfloor = 0$, i.e.\ $\varepsilon_i > \varepsilon_{p\alpha + q\beta} = p\varepsilon_\alpha + q\varepsilon_\beta \geq p\varepsilon_i$, which is a contradiction. Thus we must have $p\varepsilon_\alpha + q\varepsilon_\beta \geq 1$, whence $pm_{\alpha} + qm_{\beta} - m_{p\alpha + q\beta} = \lfloor p\varepsilon_\alpha + q\varepsilon_\beta \rfloor  \geq 1$. Thus $\breve U_{p\alpha + q\beta, pm_{\alpha} + qm_{\beta}} \subseteq \breve U_{p\alpha + q\beta, m_{p\alpha + q\beta} + 1}$, whose image in $\breve G_2^1$ vanishes. We may finally conclude that $[\bU_{\alpha,2}^1, \bU_{\beta,2}^1] \subseteq H(\varepsilon_i)$, which finishes the proof of normality of $H(\varepsilon_i)$ in $\bG_2^1$.

For $\alpha$ and $\beta$ non-reductive with $\varepsilon_{\alpha} = \varepsilon_{\beta} = \varepsilon_i$, a similar computation shows that $[\bU_{\alpha,2}^1, \bU_{\beta,2}^1] \subseteq H(\varepsilon_{i+1})$. Thus $H(\varepsilon_i)/H(\varepsilon_{i+1})$ is abelian. \qedhere
\end{proof}

\subsection{Pairings induced by the commutator}\label{sec:preparations_coh_Sigmaprime}

Let $N,N^-$ be the unipotent radicals of any two opposite Borel subgroups of $G$ which contain $T$ and are defined over $\breve k$. (We will specify $N$ to suit our needs in Section \ref{sec:hatSigmaprime}.) For $r-1 \geq a \geq 1$, let $\bN_r$, $\bN_r^-$ and $\bN_r^a$, $\bN_r^{-,a}$ be the corresponding subgroups of $\bG_r$ and $\bG_r^a$. Let $\Phi^+ = \{ \alpha \in \Phi \colon \bU_{\alpha,r} \subseteq \bN_r \}$ and $\Phi^- = \Phi \sm \Phi^+ = \{ \alpha \in \Phi \colon \bU_{\alpha,r} \subseteq \bN_r^- \}$.  For $s+1 \geq i \geq 1$, set $\Phi_i^+ = \Phi_i \cap \Phi^+$ and $\Phi_i^- = \Phi_i \cap \Phi^-$, and let $\bN_r^{1,i} = \bG_r^{1,i} \cap \bN_r$. We study some pairings induced by the commutator map. Note that the targets of the maps in Lemma \ref{lm:commutator_pairings} are abelian by Lemma \ref{lm:g_filtration_abelian}.

\begin{lm}\label{lm:commutator_pairings} Let $r\geq 2$ and $r-1 \geq a \geq 1$. Let $\alpha \in \Phi$ be a non-reductive root.
\begin{itemize}
\item[(i)] Let $a \geq 2$. The commutator map induces a bilinear pairing of abelian groups,
\[
\bU_{\alpha,r}^{r-a}/ \bU_{\alpha,r}^{r-a+1} \times \bN_r^a/\bN_r^{a+1} \rightarrow \bG_r^{r-1}, \quad (\bar{\xi}, \bar{x}) \mapsto [\bar{\xi},\bar{x}].
\]
\item[(ii)] Let $a = 1$ and $s + 1 \geq i \geq 1$. Assume that $\varepsilon_{-\alpha} = \varepsilon_i$ (thus $\varepsilon_\alpha = 1 - \varepsilon_i$). We have $[\bU_{\alpha,r}^{r-1}, \bN_r^{1,i}] \subseteq \bG_r^{r-1, {s+1}}$ and $[\bU_{\alpha,r}^{r-1}, \bN_r^{1,i+1}] = 1$. The commutator map induces a bilinear pairing of abelian groups,
\[
\bU_{\alpha,r}^{r-1} \times \bN_r^{1,i}/\bN_r^{1,i+1} \rightarrow \bG_r^{r-1,{\rm s+1}}, \quad (\xi, \bar{x}) \mapsto [\bar{\xi},\bar{x}]. 
\]

\end{itemize}
\end{lm}

\begin{proof}
(i): By Lemma \ref{lm:comm} applied three times, the commutator map $\bU_{\alpha,r}^{r-a} \times \bN_r^a\rightarrow \bG_r$ induces the claimed pairing. It is linear in $\bar{x}$: if $x_1,x_2 \in \breve N_r^a$, then
\begin{align*}
[\xi,x_1x_2] = \xi^{-1} x_2^{-1}x_1^{-1}\xi x_1 x_2 = \xi^{-1} x_1^{-1} x_2^{-1} \xi x_2 x_1 = \xi^{-1} x_1^{-1} \xi [\xi,x_2] x_1 = [\xi,x_1] [\xi,x_2],
\end{align*}
where the second equality follows from Lemma \ref{lm:comm} and $\bN_r^a/\bN_r^{a+1}$ being abelian, and the fourth follows from Lemma \ref{lm:comm} as $[\xi,x_2] \in \breve G_r^{r-1}$, the assumption $a \geq 2$, and the subsequent fact that $\bN_{a+1}^a$ is generated by root subgroups contained in it. The linearity in $\bar{\xi}$ is shown similarly.

(ii): 
We work on $\overline \FF_q$-points. To show the first claim, we observe that $\bU_{\alpha,r}^{r-1}$ commutes with $\bN_r^2$ by Lemma \ref{lm:comm}. As $\bN_r^{1,i+1}$ is generated by $\bN_r^2$ along with $\bU_{\beta,r}^{1}$ for all $\beta$ which are either reductive or satisfy $\varepsilon_{\beta} \geq \varepsilon_i$, we have to show that $[\bU_{\alpha,r}^{r-1}, \bU_{\beta,r}^1] \subseteq \bG_r^{r-1,{s+1}}$ for all such $\beta$. We have two cases:

\medskip
\noindent\emph{Case: $\beta$ is non-reductive.} We have to show that $[\breve U_{\alpha,m_{\alpha} + r-2}, \breve U_{\beta,m_{\beta}}]$ maps to $\breve G_r^{r-1,s+1}$ inside $\breve G_r$. Using Lemma \ref{lm:usual_commutator_relations}(ii), it is enough to show that for all $p,q \in \bZ_{\geq 1}$ such that $p\alpha + q\beta \in \Phi$, $\breve U_{p\alpha + q\beta, p(m_{\alpha} + r - 2) + qm_{\beta}}$ maps to $1$ in $\breve G_r$ if $p\alpha + q\beta$ is non-reductive and maps to $\breve U_{p\alpha + q\beta,r}^{r-1}$ if $p\alpha + q\beta$ is reductive. In both cases, this amounts to the claim that
\[p(m_\alpha + r - 2) + qm_{\beta} \geq m_{p\alpha + q\beta} +r - 1, \]
which in turn by Lemma \ref{lm:epsilons_and_ms} is equivalent to 
\[ \lfloor p\varepsilon_\alpha + q\varepsilon_\beta \rfloor  + (p-1)(r-2) \geq 1, \]
which is true as $\varepsilon_\beta \geq  \varepsilon_i = \varepsilon_{-\alpha} = 1-\varepsilon_{\alpha}$. 

\medskip
\noindent\emph{Case: $\beta$ is reductive.} This case is shown similarly (in fact, slightly simplier) to the above, and we omit the details. This finishes the proof of the first claim, i.e., $[\bU_{\alpha,r}^{r-1}, \bN_r^{1,i}] \subseteq \bG_r^{r-1, s+1}$. 

\mbox{}

We now show the second claim, i.e., $[\bU_{\alpha,r}^{r-1}, \bN_r^{1,i+1}] = 1$. Proceeding analogously as in the proof of the first claim, we need only to show that for all $\beta \in \Phi$ either reductive or satisfying $\varepsilon_\beta \geq \varepsilon_{i+1}$, one has $[\bU_{\alpha,r}^{r-1},\bU_{\beta,r}^1] = 1$. We again have two cases:


\medskip
\noindent\emph{Case: $\beta$ is non-reductive.} We have to show that $[\breve U_{\alpha,m_{\alpha} + r-2}, \breve U_{\beta,m_{\beta}}]$ maps to $1$ in $\breve G_r$. Using Lemma \ref{lm:usual_commutator_relations}(ii), it is enough to show that for all $p,q \in \bZ_{\geq 1}$ such that $p\alpha + q\beta \in \Phi$, $\breve U_{p\alpha + q\beta, p(m_{\alpha} + r - 2) + qm_{\beta}}$ maps to $1$ in $\breve G_r$. If $p\alpha + q\beta$ is non-reductive, this follows from the similar statement in the proof of the first claim, as $\varepsilon_{i+1} \geq \varepsilon_i$. If $p\alpha + q\beta$ is reductive, it amounts to claim that
\[p(m_\alpha + r - 2) + qm_{\beta} \geq m_{p\alpha + q\beta} +r, \] 
which by Lemma \ref{lm:epsilons_and_ms} is equivalent to 
\[ \lfloor p\varepsilon_\alpha + q\varepsilon_\beta \rfloor  + (p-1)(r-2) \geq 2, \]
But this is true, as $\lfloor p\varepsilon_\alpha + q\varepsilon_\beta \rfloor \geq 2$. Indeed, as $p\alpha + q\beta$ is reductive, $\varepsilon_{p\alpha+ q\beta} = 0$. Hence by Lemma \ref{lm:epsilons_and_ms} $\lfloor p\varepsilon_\alpha + q\varepsilon_\beta \rfloor = p\varepsilon_\alpha + q\varepsilon_\beta \geq \varepsilon_\alpha + \varepsilon_\beta > 1$. Being an integer, $\lfloor p\varepsilon_\alpha + q\varepsilon_\beta \rfloor$ must be $\geq 2$.

\medskip
\noindent \emph{Case: $\beta$ is reductive.} This case is shown similarly (in fact, slightly simpler) to the above, and we omit the details. This finishes the proof of the second claim. 

\mbox{}

We are now ready to show that the claimed pairing is well-defined. Indeed, let $\xi \in \breve U_{\alpha,r}^{r-1}$ and let $x,x' \in \breve N_r^{1,i}$ with the same image $\bar x = \bar x' \in \breve N_r^{1,i}/\breve N_r^{1,i+1}$. Then there is an $y \in \breve N_r^{1,i+1}$ such that $x' =xy$. We compute:
\begin{align*}
[\xi,x'] = [\xi,xy] = \xi^{-1}y^{-1}x^{-1} \xi xy = y^{-1} [\xi,x] y = [\xi,x],
\end{align*}
where for the third equality we use that $[\bU_{\alpha,r}^{r-1}, \bN_r^{1,i+1}] = 1$ and for the last we use that $[\xi,x] \in \breve G_r^{r-1,s+1}$ and $[\bG_r^{r-1,s+1}, \bN_r^1] = 1$ (indeed, for any reductive root $\gamma$ we have $[\bU_{\gamma,r}^{r-1}, \bN_r^1] = 1$ by Lemma \ref{lm:comm}). Now we show that this pairing is linear in the second variable. Therefore, let $\xi \in \breve U_{\alpha,r}^{r-1}$ and $x_1,x_2 \in \breve N_r^{1,i}$. We compute:
\begin{align*}
[\xi,x_1x_2] &= \xi^{-1}x_2^{-1}x_1^{-1} \xi x_1x_2  = \xi^{-1}[x_2,x_1] x_1^{-1}x_2^{-1} \xi x_1 x_2 \\ 
&= [x_2,x_1]\xi^{-1} x_1^{-1}x_2^{-1}\xi x_1x_2 = [x_2,x_1] \xi^{-1} x_1^{-1} x_2^{-1} \xi x_2 x_1 [x_1,x_2] \\ 
&= [x_2,x_1]\xi^{-1}x_1^{-1}\xi [\xi,x_2] x_1 [x_1,x_2] = [x_2,x_1] [\xi,x_1][\xi,x_2][x_1,x_2] \\
&= [\xi,x_1][\xi,x_2].
\end{align*}
The third equality follows as $[x_2,x_1] \in \breve N_r^{1,i+1}$ (as $\bN_r^{1,i}/\bN_r^{1,i+1}$ is abelian) and as $[\bU_{\alpha,r}^{r-1}, \bN_r^{1,i+1}] = 1$. The sixth equality follows as $[\xi,x_2] \in \breve G_r^{r-1,s+1}$ commutes with $x_1 \in \breve N_r^1$. The last equality follows as $[\xi,x_1], [\xi,x_2] \in \breve G_r^{r-1, {\rm s+1}}$ commute with $[x_1,x_2] \in \breve N_r^1$, and as $[x_2,x_1][x_1,x_2] = 1$. An analogous (slightly simplier) computation shows the linearity in the first variable.
\end{proof}

\begin{rem} Lemma \ref{lm:usual_commutator_relations}(ii) can certainly be generalized. As we will not use the following generalization, we state it without proof. As for any root $\alpha \in \Phi$, $-\alpha$ is a root too, and $\varepsilon_{-\alpha} = 1-\varepsilon_{\alpha}$, we have a symmetry between the jumps $\varepsilon_i$. Concretely, we have $\varepsilon_i = 1 - \varepsilon_{s+1-i}$ for $1\leq i \leq s$. For each $1\leq a \leq r-1$, let $\bG_r^{a,i}$ be the subgroup of $\bG_r^a$ generated by $\bG_r^{a+1}$, $\bT_r^a$, $\bU_{\alpha,r}^a$ ($\alpha$ reductive or $\varepsilon_{\alpha} \geq \varepsilon_i$). Then Lemma \ref{lm:usual_commutator_relations} extends to the following general duality statement: Fix $1\leq a\leq r-1$ and  $1\leq i \leq s$. Then the commutator induces a bilinear pairing,
\[
\bG_r^{r-a, s+1-i}/\bG_r^{r-a,s+2-i} \times \bG_r^{a,i}/\bG_r^{a,i+1} \rar \bG_r^{r-1,s+1}. \tag*{$\Diamond$}
\]
\end{rem}

\subsection{Stratification on (subgroups of) $\bN_r^1$}\label{sec:stratification_on_Kr1}


Recall that for any subgroup $H \subset G$ and associated subgroups $\bH_r \subset \bG_r$, we have the notation $\bH_r^{a,*} = \bH_r^a \smallsetminus \bH_r^{a+1}$ (open subscheme) and hence the corresponding set $\breve H_r^{a,*}$ of $\overline{\FF}_q$-valued points.

\begin{lemma}\label{lm:Az}
Let $r \geq 2$ and let $r-1 \geq a \geq 1$. For $z \in \breve N_r^{a,*}$, write $z = \prod_{\beta \in \Phi^+} x_\beta^z$ with $x_\beta^z \in \breve U_{\beta,r}^a$ for a fixed (but arbitrary) order on $\Phi^+$. For $\beta \in \Phi^+$, let $a \leq a(\beta, z) \leq r$ be the integer such that $x_\beta^z \in \breve U_{\beta, r}^{a(\beta,z),*}$.
\begin{enumerate}[label=(\roman*)]
\item
If $a \geq 2$, then the set
\begin{equation*}
A_z \colonequals \{\beta \in \Phi^+ : a(\beta, z) = a\}
\end{equation*}
is non-empty and independent of the chosen order on $\Phi^+$.

\item
Let $a=1$ 
and let $s+1 \geq i \geq 1$ be such that $z \in \breve N_r^{1,i,\ast}$. Then the set
\begin{equation*}
A_z \colonequals \{\beta \in \Phi^+_i \colon a(\beta,z) = 1 \}
\end{equation*}
is non-empty and independent of the chosen order on $\Phi^+$. Moreover, $a(\beta, z) > 1$ for all $\beta \in \bigcup_{j=1}^{i-1} \Phi_j^+$.
\end{enumerate}
\end{lemma}

\begin{proof}
(i): As $a\geq 2$, the quotient $\bN_r^a/\bN_r^{a+1}$ is abelian by Lemma \ref{lm:g_filtration_abelian}. Thus its $\overline{\FF}_q$-points are simply tuples $(\bar x_{\beta})_{\beta \in \Phi^+}$ with $\bar x_{\beta} \in \breve U_{\beta,a+1}^a$ with entry-wise multiplication. If $\bar{z}= (\bar{x}_{\beta}^z)$ is the image of $z$ in this quotient, then $A_z$ identifies with the set of those $\beta$ for which $\bar{x}_{\beta}^z \neq 1$ (which is obviously independent of the order). 

(ii): Assume that the last claim of (ii) is not true. Then let $1\leq i_0 < i$ be the smallest integer such that $a(\beta,z) = 1$ for some $\beta \in \Phi_{i_0}^+$. Then from Lemma \ref{lm:g_filtration_abelian} it follows that $z \in \breve N_r^{1,i_0,\ast}$, which contradicts the assumption. This shows the last claim. 
The first claim follows by the same argument as in (i). \qedhere
\end{proof}

Using Section \ref{sec:preparations_coh_Sigmaprime} we can now prove the following generalization of \cite[Lemma 1.7]{Lusztig_04}. 

\begin{definition}
For $\alpha \in \Phi^+$ define its \textit{height} $\height(\alpha)$ (relative to $N$) to be the largest integer $m\geq 1$ such that $\alpha = \sum_{i=1}^m \alpha_i$ with $\alpha_i \in \Phi^+$.  
\end{definition}

\begin{prop}\label{prop:commutator_producing_torus_element}
Let $r\geq 2$ and let $r - 1 \geq a \geq 1$. Let $z = \prod_{\beta \in \Phi^+} x_\beta^z \in \breve N_r^{a,\ast}$ for $x_\beta^z \in \breve U_{\beta,r}^a$ and let $A_z$ be as in Lemma \ref{lm:Az}.
\begin{enumerate}[label=(\roman*)]
\item If $A_z$ contains a non-reductive root, let $-\alpha \in A_z$ be a non-reductive root of maximal height and $\alpha \in \Phi^-$ its opposite. Then for any $\xi \in \bU_{\alpha,r}^{r-a}$, we have $[\xi,z] \in \mathcal{T}^{\alpha}\bN_r^{-,r-1}$. Moreover, projecting $[\xi,z]$ into $\mathcal{T}^{\alpha}$ induces an isomorphism
\[
\lambda_z \colon \bU_{\alpha,r}^{r-a}/\bU_{\alpha,r}^{r-a+1} \stackrel{\sim}{\rightarrow} \mathcal{T}^{\alpha}
\]
\item If $A_z$ contains only reductive roots, let $-\alpha \in A_z$ be a root of maximal height and $\alpha \in \Phi^-$ its opposite. Then for any $\xi \in \bU_{\alpha,r}^{r-a-1}$, we have $[\xi,z] \in \mathcal{T}^{\alpha}\bN_r^{-,r-1}$. Moreover, projecting $[\xi,z]$ into $\mathcal{T}^{\alpha}$ induces an isomorphism
\[
\lambda_z \colon \bU_{\alpha,r}^{r-a-1}/\bU_{\alpha,r}^{r-a} \stackrel{\sim}{\rightarrow} \mathcal{T}^{\alpha}
\]
\end{enumerate} 
\end{prop}

\begin{proof}
Parts (i) and (ii) can be proven in the same way. We give the full proof of (i) only.

\medskip
\noindent \emph{Proof of (i) when $a \geq 2$}. We work on $\overline \FF_q$-points. Assume that $A_z$ contains a non-reductive root and let $-\alpha$ be such a root of maximal height and $\alpha \in \Phi^-$ its opposite. Let $\xi \in \breve U_{\alpha,r}^{r-a}$ and let $\bar{\xi} \in \breve U_{\alpha,r}^{r-a}/\breve U_{\alpha,r}^{r-a+1}$ and $\bar z \in \breve N_r^a/\breve N_r^{a+1}$ be the images of $\xi$ and $z$ respectively. By Lemma \ref{lm:g_filtration_abelian} we may write 
\[ 
\bar{z} = \bar{x}_{-\alpha}^z \prod_{\beta \in \Phi^+ \text{ red.}} \bar{x}_{\beta}^z \cdot  \prod_{\substack{\beta \in \Phi^+ \text{ non-red., } \beta \neq -\alpha \\ {\rm ht}(\beta) \leq {\rm ht}(-\alpha) }} \bar{x}_{\beta}^z,
\]
where $\bar x_\beta^z \in \breve U_{\beta,r}^{a}/\breve U_{\beta,r}^{a+1}$ and where the products are taken in any order. Lemma \ref{lm:commutator_pairings} shows that $[\xi,z]$ is the product of $[\bar{\xi},\bar{x}_{-\alpha}^z]$ with all the $[\bar{\xi},\bar{x}_{\beta}^z]$ for $\beta \in \Phi^+$, the product taken in any order. If $\beta$ is reductive, then $[\bar\xi,\bar x_{\beta}^z] \in [\breve U_{\alpha,r}^{r-a}, \breve U_{\beta,r}^a] = 1$ by Lemma \ref{lm:comm}. If $\beta \neq -\alpha$ is non-reductive, then by assumption ${\rm ht}(\beta) \leq {\rm ht}(-\alpha)$. The commutator $[\bar\xi,\bar x_{\beta}^z]$ is the image of an element of 
\begin{equation}\label{eq:some_commut_in_pf_of_main_lm_1}
[\breve U_{\alpha, m_{\alpha} + (r-a) - 1}, \breve U_{\beta, m_{\beta} + a - 1}] \subseteq \prod_{\substack{p,q \in \bZ_{\geq 1} \\ p\alpha + q\beta \in \Phi}} \breve U_{p\alpha + q\beta, pm_{\alpha} + q m_{\beta} + p(r-a-1) + q(a-1)}
\end{equation}
\begin{lm}\label{lm:a2_commutator_of_nonreds}
The image of the right hand side of \eqref{eq:some_commut_in_pf_of_main_lm_1} in $\breve G_r$ lies in $\breve N_r^{-, r-1}$. 
\end{lm}

\begin{proof}
It is enough to show that for each $(p,q)$ occurring in the product, the corresponding factor is either contained in $\breve N_r^{-, r-1}$ or vanishes in $\breve G_r$. If $p\geq q$, then ${\rm ht}(\beta) \leq {\rm ht}(-\alpha)$ implies $p\alpha + q\beta \not\in \Phi^+$. So, we may assume that $q>p$ and in particular $q \geq 2$. It is enough to show that
\[
\breve U_{p\alpha + q\beta, pm_{\alpha} + q m_{\beta} + p(r-a-1) + q(a-1)} \subseteq \begin{cases} \breve U_{p\alpha + q\beta, m_{p\alpha + q\beta} + r} & \text{if $p\alpha + q\beta$ reductive} \\ \breve U_{p\alpha + q\beta, m_{p\alpha + q\beta} + r - 1} & \text{otherwise,} \end{cases}
\]
as both map to $1$ in $\breve G_r$. Equivalently, we have to show that
\[
pm_{\alpha} + qm_{\beta} - m_{p\alpha + q\beta} + p(r-a-1) + q(a-1) - (r-1) \geq \begin{cases} 1 & \text{if $p\alpha + q\beta$ reductive} \\ 0 & \text{otherwise.} \end{cases}
\]
But this holds as by Lemma \ref{lm:epsilons_and_ms}, $pm_{\alpha} + qm_{\beta} - m_{p\alpha + q\beta} = \lfloor p \varepsilon_\alpha + q\varepsilon_\beta \rfloor$ is $\geq 1$ if $p\alpha + q\beta$ is reductive and is $\geq 0$ otherwise, and as $q\geq 2$ and $a\geq 2$.
\end{proof}

Finally, $[\bar{\xi},\bar{x}_{-\alpha}^z] = [\xi,x_{-\alpha}^z] \in \cT^{\alpha}(\overline \FF_q)$ by Lemma \ref{lm:usual_commutator_relations}(iii). Thus $[\xi,z] \in \cT^{\alpha}(\overline \FF_q)\breve N_r^{-,r-1}$. Moreover, if we project onto $\cT^{\alpha}(\overline \FF_q)$, then only $[\bar{\xi},\bar{x}_{-\alpha}^z]$ survives and Lemma \ref{lm:usual_commutator_relations}(iii) proves the desired isomorphism $\lambda_z$. This finishes the proof of (i) in the case $a \geq 2$.

\medskip
\noindent \emph{Proof of (i) when $a = 1$}. Let $s \geq i \geq 1$ denote the integer such that $z \in \breve N^{1,i,\ast}$. (Note that $i\neq s+1$ as $A_z$ contains a non-reductive root by assumption). We have $\xi \in \breve U_{\alpha,r}^{r-1}$, and we let $\bar z$ denote the image of $z$ in $\breve N_r^{1,i}/\breve N_r^{1,i+1}$. By Lemma \ref{lm:g_filtration_abelian} we may write
\[
\bar{z} = \bar{x}_{-\alpha}^z \left(\prod_{\substack{\beta \in \Phi^+_i \colon \beta \neq -\alpha \\ {\rm ht}(\beta) \leq {\rm ht}(-\alpha)}} \bar{x}_{\beta}^z\right), 
\]
(product are taken in any order). By Lemma \ref{lm:commutator_pairings}, $[\xi,\bar z]$ is the product of $[\xi,\bar{x}_{-\alpha}^z]$ with all the $[\xi, \bar{x}_{\beta}^z]$ taken in any order. By assumption $\varepsilon_\beta = \varepsilon_i = \varepsilon_{-\alpha} = 1 -  \varepsilon_{\alpha}$. In particular, all $\beta$'s are non-reductive. Now, $[\xi,\bar x_{\beta}^z]$ is the image in $\breve G_r^{r-1,s+1}$ of an element of 
\begin{equation}\label{eq:some_commut_in_pf_of_main_lm_a1}
[\breve U_{\alpha, m_{\alpha} + r-2}, \breve U_{\beta, m_{\beta}}] \subseteq \prod_{\substack{p,q \in \bZ_{\geq 1} \\ p\alpha + q\beta \in \Phi}} \breve U_{p\alpha + q\beta, pm_{\alpha} + q m_{\beta} + p(r-2)}
\end{equation}
\begin{lm}\label{lm:a1_commutator_of_nonreds}
The image of the right hand side of \eqref{eq:some_commut_in_pf_of_main_lm_a1} in $\breve G_r$ lies in $\breve N_r^{-, r-1}$.
\end{lm}

\begin{proof}
Note that the right hand side of \eqref{eq:some_commut_in_pf_of_main_lm_a1} is contained in $\breve G_r^{r-1,s+1}$ (exactly as in the proof of Lemma \ref{lm:commutator_pairings}(ii)). Now the same arguments as in the proof Lemma \ref{lm:a2_commutator_of_nonreds} apply. If $p\geq q$, then ${\rm ht}(\beta) \leq {\rm ht}(-\alpha)$ implies $p\alpha + q\beta \not\in \Phi^+$, thus the corresponding factor of the product is contained in $\breve N_r^- \cap \breve G_r^{r-1,s+1} \subseteq \breve N_r^{-,r-1}$. Thus we may assume that $q>p$ and in particular $q\geq 2$. It is enough to show that 
\[
\breve U_{p\alpha + q\beta, pm_{\alpha} + q m_{\beta} + p(r-2)} \subseteq \begin{cases} \breve U_{p\alpha + q\beta, m_{p\alpha + q\beta} + r}&\text{if $p\alpha + q\beta$ is reductive} \\ \breve U_{p\alpha + q\beta, m_{p\alpha + q\beta} + r -1} &\text{otherwise,} \end{cases}
\]
as both map to $1$ in $\breve G_r$. Equivalently, we have to show that
\[
pm_{\alpha} + q m_{\beta} - m_{p\alpha + q\beta} + p(r-2) - (r-1) \geq \begin{cases} 1 & \text{if $p\alpha + q\beta$ is reductive} \\ 0 &\text{otherwise.} \end{cases}
\]
By Lemma \ref{lm:epsilons_and_ms}, this follows from $\lfloor p\varepsilon_\alpha + q\varepsilon_\beta \rfloor \geq 2$ if $p\alpha + q\beta$ is reductive, resp. to $\lfloor p\varepsilon_\alpha + q\varepsilon_\beta \rfloor \geq 1$ if $p\alpha + q\beta$ is non-reductive. But in any case we have $p\varepsilon_\alpha + q\varepsilon_\beta \geq \varepsilon_\alpha + 2(1-\varepsilon_\alpha) = 2-\varepsilon_\alpha > 1$ by assumptions. In particular, we are done in the case when $p\alpha + q\beta$ is non-reductive. If $p\alpha + q\beta$ is reductive, then $p\varepsilon_\alpha + q\varepsilon_\beta$ must also be an integer (by Lemma \ref{lm:epsilons_and_ms}) and hence $\geq 2$, and we are done in this case too.
\end{proof}

Finally, $[\xi,\bar{x}_{-\alpha}^z] \in \cT^{\alpha}(\overline \FF_q)$ by Lemma \ref{lm:usual_commutator_relations}(iii). Thus $[\xi,z] \in \cT^{\alpha}(\overline \FF_q)\breve N_r^{-,r-1}$. Moreover, if we project onto $\cT^{\alpha}(\overline \FF_q)$, then only $[\bar{\xi},\bar{x}_{-\alpha}^z]$ survives and Lemma \ref{lm:usual_commutator_relations}(iii) proves the desired isomorphism $\lambda_z$. This finishes the proof of (i). 
\end{proof}

\begin{rem}
We note that in the proof of \cite[Lemma 1.7]{Lusztig_04} there is an (easily correctable) mistake. It is claimed that whenever $-\alpha, \beta \in \Phi^+$ with $-\alpha \neq \beta$ and ${\rm ht}(-\alpha) \geq {\rm ht}(\beta)$, then $p\alpha + q\beta \not\in \Phi^+$ for all $p,q \in \bZ_{\geq 1}$. This is not true. For example, let $\Phi$ be of type $C_2$, let $\epsilon_1,\epsilon_2$ denote a basis for $X^{\ast}(T)$ such that the $\Phi^+ = \{\epsilon_1 - \epsilon_2, \epsilon_1 + \epsilon_2, 2\epsilon_1, 2\epsilon_2\}$. Then taking $\alpha = -2\epsilon_1$, $\beta = \epsilon_1 + \epsilon_2$. Then ${\rm ht}(-\alpha) = 3 > 2 = {\rm ht}(\beta)$. But $\alpha + 2\beta = 2\epsilon_2 \in \Phi^+$. Observe here that $\alpha + \beta \notin \Phi^+$, which contradicts the parenthetical assertion at the end of the proof of \cite[Lemma 1.7]{Lusztig_04}.

Surely, the statement of \cite[Lemma 1.7]{Lusztig_04} remains true. The place in its proof, where the abovementioned claim is used, can be corrected as follows: if $p\alpha + q\beta \in \Phi^+$ for some $p,q \in \bZ_{\geq 1}$, then $q \geq 2$ and the part of the commutator (as in the proof of Proposition \ref{prop:commutator_producing_torus_element}) inside $\bU_{p\alpha + q\beta,r}$ vanishes, since all roots are reductive and $r \geq 2$. \hfill $\Diamond$ 
\end{rem}

Let $\bK_r = \bU_r^- \cap \bN_r$. Let $\Phi' = \{\beta \in \Phi^+ \colon \bU_{\beta,r} \subseteq \bK_r \}$. Let $\mathcal{X}$ denote the set of all non-empty subsets $I \subseteq \Phi'$ satisfying
\begin{enumerate}[label=(\roman*)]
\item
the restriction of ${\rm ht} \colon \Phi^+ \rightarrow \bZ_{\geq 0}$ to $I$ is constant, and
\item
$I$ contains either only reductive or only non-reductive roots. 
\end{enumerate}
To $z \in \breve K_r^1 \sm \{1\}$ we attach a pair $(a_z,I_z)$ with  $1 \leq a_z \leq r-1$ and $I_z \in \mathcal{X}$. Define $a_z$ by $z \in \breve K_r^{a_z,\ast}$. Let $A_z$ be as in Lemma \ref{lm:Az}. If $A_z$ contains a non-reductive root, let $I_z \subseteq A_z$ be the subset of all non-reductive roots of maximal height. If $A_z$ contains only reductive roots, let $I_z \subseteq A_z$ be the subset of all roots of maximal height. We have a stratification into locally closed subsets
\begin{equation}\label{eq:stratification_of_Kh1}
\bK_r^1 \sm \{1\} = \bigsqcup_{a,I} \bK_r^{a,\ast,I}, \qquad \text{where $\bK_r^{a,\ast,I}(\overline{\FF}_q) = \{z \in \breve K_r^1 \sm \{1\} \colon (a_z,I_z) = (a,I)\}$}.
\end{equation}

\subsection{Cohomology of $\widehat \Sigma'$}\label{sec:hatSigmaprime} 

We now complete the proof of Theorem \ref{thm:altsum_coh_Sigma} by proving Claim \eqref{e:prime}. Using the stratification \eqref{eq:stratification_of_Kh1} and Proposition \ref{prop:commutator_producing_torus_element}, Claim  \eqref{e:prime} is proven in exactly the same way as in \cite[Lemma 1.9]{Lusztig_04}. For convenience, we sketch the arguments here. It is enough to show that $H^j_c(\widehat \Sigma'_w)_{\theta,\theta'} = 0$ for all $j \geq 0$. For a $\mathcal{T}'(\overline \FF_q)^{\sigma}$-module $M$ and a character $\chi$ of $\mathcal{T}'(\overline \FF_q)^{\sigma}$, write $M_{(\chi)}$ for the $\chi$-isotypic component of $M$. Note that $\mathcal{T}'(\overline \FF_q)^{\sigma}$ acts on $\widehat\Sigma'_w$ by
\[
t' \colon (x,x',u,u',z,\tau') \mapsto (x, t'x't^{\prime-1}, u, t'u't^{\prime -1}, z, \tau't^{\prime -1}).
\]
Hence $H^j_c(\widehat \Sigma'_w)$ is a $\mathcal{T}'(\FF_q)$-module. It is enough to show that $H^j_c(\widehat \Sigma'_w)_{(\chi)} = 0$ for any regular character $\chi$ of $\mathcal{T}'(\FF_q)$. Fix such a $\chi$. Set $N = \dot w U^{\prime -} \dot w^{-1}$, $N^- = \dot w U^{\prime} \dot w^{-1}$. The stratification \eqref{eq:stratification_of_Kh1} of $\bK_r^1 \smallsetminus \{1\}$ induces a stratification of $\widehat \Sigma'_w$ into locally closed subsets indexed by $1\leq a\leq r-1$ and $I \in \mathcal{X}$:
\begin{equation*}
\widehat\Sigma'_w = \bigsqcup_{a,I} \widehat\Sigma_w^{\prime,a,I} \qquad \text{where $\widehat\Sigma_w^{\prime,a,I}(\overline{\FF}_q) = \{(x,x',u,u',z,\tau') \in \widehat\Sigma'_w(\overline{\FF}_q) \colon z\in \breve K_r^{a,\ast,I} \}$}.
\end{equation*}
Note that each $\widehat \Sigma_w^{\prime, a, I}$ is stable under $\mathcal{T}'(\FF_q)$. Thus \eqref{e:prime} follows from 
\begin{equation}\label{e:Sigma w a I}
H_c^j(\widehat\Sigma_w^{\prime,a,I}, \overline{\mathbb{Q}}_{\ell})_{(\chi)} = 0 \qquad \text{for any fixed $a,I$.}
\end{equation}
To show \eqref{e:Sigma w a I}, choose a root $\alpha$ such that $-\alpha \in I$. Then $\bU_{\alpha,r} \subseteq \bU_r \cap \dot{w} \bU_r' \dot{w}^{-1}$. By Proposition \ref{prop:commutator_producing_torus_element} for any $z \in \breve K_r^{a,\ast,I}$, we have an isomorphism 
\begin{align*}
\lambda_z &\from \bU_{\alpha,r}^{r-a}/\bU_{\alpha,r}^{r-a+1} \stackrel{\sim}{\longrightarrow} \cT^\alpha, && \text{if $\alpha$ is non-reductive,} \\
\lambda_z &\from \bU_{\alpha,r}^{r-a-1}/\bU_{\alpha,r}^{r-a} \stackrel{\sim}{\longrightarrow} \cT^\alpha, && \text{if $\alpha$ is reductive.}
\end{align*}
Let $\pi$ denote the natural projection $\bU_{\alpha,r}^{r-a} \rightarrow \bU_{\alpha,r}^{r-a}/\bU_{\alpha,r}^{r-a+1}$ if $\alpha$ is non-reductive and the natural projection $\bU_{\alpha,r}^{r-a-1} \rightarrow \bU_{\alpha,r}^{r-a-1}/\bU_{\alpha,r}^{r-a}$ if $\alpha$ is reductive. Let $\psi$ be a section to $\pi$ such that $\pi\psi = 1$ and $\psi(1) = 1$. Let
\[
\mathcal{H}' \colonequals \{t' \in \mathcal{T}' \colon t^{\prime -1} \sigma(t') \in \dot{w}^{-1} \mathcal{T}^{\alpha} \dot{w} \}.
\]
This is a closed subgroup of $\mathcal{T}'$. For any $t' \in \mathcal{T}'$ define $f_{t'} \colon \widehat\Sigma_w^{\prime,a,I} \rightarrow \widehat\Sigma_w^{\prime,a,I}$ by
\[f_{t'}(x,x',u,u',z,\tau') = (x\sigma(\xi), \hat{x}^{\prime}, u, \sigma(t')^{-1} u' \sigma(t'), z, \tau' \sigma(t')), \]
where
\[
\xi = \psi\lambda_z^{-1}(\dot{w}\sigma(t')^{-1}t'\dot{w}^{-1}) \in \begin{cases} \bU_{\alpha,r}^{r-a-1} \subseteq \bU_r \cap \dot{w}\bU_r'\dot{w}^{-1} &\text{if $\alpha$ is reductive,}\\ \bU_{\alpha,r}^{r-a} \subseteq \bU_r \cap \dot{w}\bU_r'\dot{w}^{-1} & \text{otherwise,} \end{cases}
\]
and $\hat{x}' \in \bG_r$ is defined by the condition that
\[
x\sigma(\xi z \dot{w} \tau' \sigma(t')) \in u z \dot{w} \tau' \sigma(t')\sigma(t')^{-1}u'\sigma(t')\hat{x}'. 
\]
To check that $f_{t'}$ is well-defined we have to show $\hat{x}' \in \sigma(\bU_r')$. This is done with exactly the same computation as in the proof of \cite[Lemma 1.9]{Lusztig_04}, and we omit this. It is clear that $f_{t'} \colon \widehat\Sigma_w^{\prime,a,I} \rightarrow \widehat\Sigma_w^{\prime,a,I}$ is an isomorphism for any $t' \in \mathcal{H}'$. Moreover, since $\mathcal{T}'(\FF_q) \subseteq \mathcal{H}'$ and since for any $t' \in \mathcal{T}'(\FF_q)$ the map $f_{t'}$ coincides with the action of $t'$ in the $\mathcal{T}'(\FF_q)$-action on $\widehat\Sigma_w^{\prime,a,I}$ (we use $\psi(1) = 1$ here), it follows that we have constructed an action $f$ of $\mathcal{H}'$ on $\widehat\Sigma_w^{\prime,a,I}$ extending the $\mathcal{T}'(\FF_q)$-action.

If a connected group acts on a scheme, the induced action in the cohomology is constant. Thus for any $t' \in \mathcal{H}^{\prime \circ}$, the induced map $f_{t'}^{\ast} \colon H_c^j(\widehat\Sigma_w^{\prime,a,I}, \overline{\mathbb{Q}}_{\ell}) \rightarrow H_c^j(\widehat\Sigma_w^{\prime,a,I}, \overline{\mathbb{Q}}_{\ell})$ is constant when $t'$ varies in $\mathcal{H}^{\prime \circ}$. Hence $\mathcal{T}'(\FF_q) \cap \mathcal{H}^{\prime\circ}$ acts trivially on $H_c^j(\widehat\Sigma_w^{\prime,a,I}, \overline{\mathbb{Q}}_{\ell})$.

We can find some $m\geq 1$ such that $\sigma^m(\dot{w}^{-1}\mathcal{T}^{\alpha}\dot{w}) = \dot{w}^{-1}\mathcal{T}^{\alpha} \dot{w}$. Then 
\[
t' \mapsto t'\sigma(t')\sigma^2(t') \cdots \sigma^{m-1}(t') 
\]
defines a morphism $\dot{w}^{-1}\mathcal{T}^{\alpha}\dot{w} \rightarrow \mathcal{H}'$. Since $\mathcal{T}^{\alpha}$ is connected, its image is also connected and hence contained in $\mathcal{H}^{\prime \circ}$. If $t' \in (\dot{w}^{-1}\mathcal{T}^{\alpha}(\overline \FF_q)\dot{w})^{\sigma^m}$, then $N_{\sigma}^{\sigma^m}(t') \in \mathcal{T}'(\overline \FF_q)^{\sigma}$ and hence also $N_{\sigma}^{\sigma^m}(t') \in \mathcal{T}^{\prime}(\overline \FF_q)^{\sigma} \cap \mathcal{H}^{\prime \circ}(\overline \FF_q)$. Thus the action of  $N_{\sigma}^{\sigma^m}(t') \in \mathcal{T}^{\prime}(\overline \FF_q)^{\sigma}$ on $H_c^j(\widehat\Sigma_w^{\prime,a,I})$ is trivial for any $t' \in (\dot{w}^{-1}\mathcal{T}^{\alpha}(\overline \FF_q)\dot{w})^{\sigma^m}$. 

Finally, observe that if $H_c^j(\widehat\Sigma_w^{\prime,a,I}, \overline{\mathbb{Q}}_{\ell})_{(\chi)} \neq 0$, then the above shows that $t' \mapsto \chi(N_{\sigma}^{\sigma^m}(t'))$ must be the trivial character, which contradicts the regularity assumption on $\chi$. This establishes \eqref{e:Sigma w a I}, and hence also \eqref{e:prime}, which was the last outstanding claim in the proof of Theorem \ref{thm:altsum_coh_Sigma}.

\section{Representations of parahoric subgroups of $G(k)$}\label{sec:representations}

\subsection{The schemes $S_{T,U}$}

Let the notation be as at the beginning of Section \ref{sec:Sigma}. In particular, fix $r \geq 1$, a $k$-rational $\breve k$-split maximal torus $T$ in $G$, a point ${\bf x} \in \sA_{T,\breve k} \cap \sB_k$, and the unipotent radical $U$ of a Borel subgroup of $G_{\breve k}$, defined over $\breve k$ and containing $T_{\breve k}$. Let $d$ be the smallest positive integer such that $\sigma^d(U) = U$. To this data, we attach the $\FF_{q^d}$-subscheme of $\bG_r$
\begin{equation*}
S_{{\bf x}, T,U,r} \colonequals \{ x \in \bG_r \colon x^{-1}\sigma(x) \in \bU_r \}. 
\end{equation*}
To match the notation of \cite{Lusztig_04}, we write $S_{T,U}$ for $S_{{\bf x}, T, U, r}$.

\begin{lm} \label{lm:STU_prop}
$S_{T,U}$ is separated, smooth, and of finite type over $\FF_{q^d}$. Moreover, $\dim S_{T,U} = (r-1)\# \Phi^+ + \# \Phi^{+, \rm red}$, where $\Phi^+$ and $\Phi^{+,\rm red}$ are the roots and reductive roots of $T$ in $U$.
\end{lm}
\begin{proof}
Indeed, $S_{T,U}$ is the pullback of $\bU_r$ under the finite \'etale Lang map $\bG_r \rar \bG_r$, $x \mapsto x^{-1} \sigma(x)$, and $\dim \bU_r = (r-1)\# \Phi^+ + \# \Phi^{+, \rm red}$ is immediate.
\end{proof}

The finite group $\breve G_r^\sigma \times \breve T_r^\sigma$ acts on $S_{T,U}$ by $(g,t) \colon x \mapsto gxt$. 

\begin{rem}
$S_{T,U}$ admits also a natural (free) action of $\bU_r \cap \sigma^{-1}(\bU_r)$ by right multiplication. If $r = 1$, the quotient of $S_{T,U}$ by this action is ($\overline{\FF}_q$-isomorphic to) a classical Deligne--Lusztig variety for the reductive $\FF_q$-group $\bG_1$. \hfill $\Diamond$
\end{rem}

\begin{lm}\label{lm:Sigma_STU}
For $(T,U)$ and $(T',U')$ as in Section \ref{sec:Sigma}, the map 
\[ \breve G_r^{\sigma} \backslash (S_{T,U} \times S_{T',U'}) \stackrel{\sim}{\rar} \Sigma, \quad (g,g') \mapsto (g^{-1} \sigma(g), g^{\prime -1}\sigma(g'), g^{-1} g'),
\]
is a $\breve T_r^{\sigma}\times \breve T_r^{\prime \sigma}$-equivariant isomorphism, where $\breve G_r^{\sigma}$ acts diagonally on $S_{T,U} \times S_{T',U'}$.
\end{lm}

By functoriality of cohomology, the $\breve G_r^\sigma \times \breve T_r^\sigma$-action on $S_{T,U}$ induces for each $i \in \bZ$ a $\breve G_r^\sigma \times \breve T_r^\sigma$-action on $H_c^i(S_{T,U}, \cool)$. For a character $\theta \colon \breve T_r^\sigma \rar \cool^{\times}$, let $H_c^i(S_{T,U}, \cool)_{\theta}$ denote the $\theta$-isotypic component. It is stable under the action of $\breve G_r^{\sigma}$. 

\begin{Def}\label{Def:RTtheta}
We define the virtual  $\breve G_r^{\sigma}$-representation with $\cool$-coefficients 
\[
R_{{\bf x}, T,U,r}^\theta \colonequals \sum_{i\in \bZ} (-1)^i H_c^i(S_{T,U}, \cool)_{\theta}.
\]
By pullback, we can also consider $R_{{\bf x}, T, U, r}^\theta$ a virtual representation of the parahoric subgroup $\breve P_{\bf x}^\sigma$ of $G(k)$. If ${\bf x}$ is clear from the context, we write $R_{T,U,r}^\theta$ for $R_{{\bf x}, T,U,r}^\theta$. 

Moreover, by Corollary \ref{cor:indep_of_U_irreduciblity} below, $R_{T,U,r}^{\theta}$ does not depend on the choice of $U$, if $\theta$ is regular. In this case we denote $R_{T,U,r}^\theta$ by $R_{T,r}^\theta$. For the dependence on $r$ see Section \ref{sec:level_compat}.
\end{Def}

Recall the group $N_{\bG_r}(\bT_r,\bT'_r)$ from Section \ref{sec:Weyl_Bruhat}. The next few results are all corollaries of Theorem \ref{thm:altsum_coh_Sigma}, generalizing \cite[Propositions 2.2, 2.3, Corollary 2.4]{Lusztig_04}.

\begin{prop}\label{prop:Lusztig_22}
Assume that $r \geq 2$. Let $(T,U)$, $(T',U')$ be as in Section \ref{sec:Sigma}. Further, let $\theta \from \breve T_r^\sigma \to \overline \QQ_\ell^\times$, $\theta' \from \breve T_r^{\prime \sigma} \to \overline \QQ_\ell^\times$ be two characters. 
\begin{itemize}
\item[(i)] Let $i,i' \in \bZ$. Assume that an irreducible $\breve G_r^\sigma$-representation appears in the dual space $(H_c^i(S_{T,U}, \cool)_{\theta^{-1}})^{\vee}$ of $H_c^i(S_{T,U}, \cool)_{\theta}$ and in $H_c^{i'}(S_{T',U'}, \cool)_{\theta'}$. Then there exists an integer $n \geq 1$ and a $g \in N_{\bG_r}(\bT_r',\bT_r)(\FF_{q^n})$ such that the adjoint action of $g$ carries $\theta \circ N_\sigma^{\sigma^n}|_{\cT(\overline \FF_q)^{\sigma^n}}$ to $\theta' \circ N_\sigma^{\sigma^n}|_{\cT'(\overline \FF_q)^{\sigma^n}}$. 
\item[(ii)] Assume that an irreducible $\breve G_r^\sigma$-representation occurs in $R_{T,U,r}^\theta$ and $R_{T',U',r}^{\theta'}$. Then there exist some $n \geq 1$ and $g \in N_{\bG_r}(\bT_r',\bT_r)(\FF_{q^n})$ such that the adjoint action of $g$ carries $\theta \circ N_\sigma^{\sigma^n}|_{\cT(\overline \FF_q)^{\sigma^n}}$ to $\theta' \circ N_\sigma^{\sigma^n}|_{\cT'(\overline \FF_q)^{\sigma^n}}$.
\end{itemize}
\end{prop}

\begin{proof}
The proof (using Lemma \ref{lm:Sigma_STU} and Lemma \ref{lm:intertwiner_of_characters_nonzero_coh}) is literally the same as the proof of \cite[Proposition 2.2]{Lusztig_04}. We omit the details.
\end{proof}

\begin{prop}\label{prop:Lusztig_23}
Assume that $r \geq 2$. Let $(T,U,\theta)$ and $(T',U',\theta)$ be two triples as in Proposition \ref{prop:Lusztig_22}. Assume that $\theta$ or $\theta'$ is regular. Then 
\[ \langle R_{T,U,r}^\theta, R_{T',U',r}^{\theta'} \rangle = \#\{ w \in W_{\bf x}(T,T')^{\sigma} \colon \theta \circ {\rm Ad}(w) = \theta' \}.
\]
\end{prop}
\begin{proof}
A standard computation using Lemma \ref{lm:Sigma_STU} and the K\"unneth formula shows that \\ $\langle R_{T,U}^\theta, R_{T',U'}^{\theta'} \rangle = \sum_{i \in \bZ} \dim H_c^i(\Sigma, \cool)_{\theta^{-1},\theta'}$. Now apply Theorem \ref{thm:altsum_coh_Sigma}.
\end{proof}

\begin{cor}\label{cor:indep_of_U_irreduciblity}
Assume that $r\geq 2$. Let $\theta \colon \breve T_r^{\sigma} \rar \cool^{\times}$ be regular.
\begin{itemize}
\item[(i)] $R_{T,U,r}^{\theta}$ is independent of the choice of $U$.
\item[(ii)] If additionally the stabilizer of $\theta$ in $W_{\bf x}(T,T)^{\sigma}$ is trivial, then $\pm R_{T,U,r}^{\theta}$ is an irreducible representation of $\breve G_r^\sigma$ and of $P_{\bf x}(\cO_k)$.
\end{itemize}
\end{cor}
\begin{proof}
See \cite[Corollary 2.4]{Lusztig_04}. 
\end{proof}

\subsection{Change of level}\label{sec:level_compat}

One could hope that if $\theta$ is a character of $T(\cO_k) = (\breve{T}^0)^\sigma$ which is trivial on $(\breve T^r)^\sigma$, then the representations $R_{T,U,r}^\theta$ and $R_{T,U,s}^\theta$ for all $s\geq r$ coincide. In \cite[Proposition 7.6]{CI_ADLV}, it is shown that this holds when $G$ is an inner form of $\GL_n(k)$ and $T$ is an elliptic torus. We will show in subsequent work that for general $G$ which split over $\breve k$, this is true when $T$ is elliptic. However this fails for general $T$. In some sense, the more $T$ splits, the bigger is the discrepancy between $R_{T,U,r}^\theta$ and $R_{T,U,r+1}^\theta$. We will explain the failure in an example. 

Assume that $G$ is quasi-split over $k$ and let $T \subseteq G$ be a maximal $k$-rational torus, which contains a $k$-split maximal torus of $G$. Under these assumptions there is a $k$-rational Borel subgroup of $G$ containing $T$. Let $U$ be its unipotent radical. There is a hyperspecial vertex ${\bf x} = {\bf x}_0$ contained in $\sA_{T, \breve k} \cap \mathscr{B}_k$. 
Let $r\geq 1$, and let $\theta$ be a character of $(\breve{T}^0)^\sigma$, which factors through the character (again denoted $\theta$) of $\breve T_r^{\sigma}$. For each $s\geq r$, 
\begin{equation*}
S_{{\bf x}, T,U,s}/\bU_s = (\bG_s/\bU_s)^{\sigma} = \bG_s^\sigma/\bU_s^\sigma
\end{equation*}
is a discrete point set. For a surjection of groups $H \tar K$, let $\Inf_K^H$ denote the inflation functor from virtual $K$-representations to virtual $H$-representations given by pullback. Since $S_{{\bf x}, T, U, s}$ and $S_{{\bf x}, T, U, s}/\bU_r$ have the same cohomology groups up to an even degree shift, we then have
\begin{align*}
R_{{\bf x},T,U,s}^{\theta} &= \Ind_{\breve B_s^\sigma}^{\breve G_s^\sigma} \Inf_{\breve B_r^\sigma}^{\breve B_s^\sigma} \theta, \\
\Inf_{\breve G_r^\sigma}^{\breve G_s^\sigma} R_{{\bf x},T,U,r}^{\theta} &= \Inf_{\breve G_r^\sigma}^{\breve G_s^\sigma} \Ind_{\breve B_r^\sigma}^{\breve G_r^\sigma} \theta = \Ind_{\breve B_s^\sigma \breve G_s^{r,\sigma}}^{\breve G_s^\sigma} \Inf_{\breve B_r^\sigma}^{\breve B_s^\sigma \breve G_s^{r,\sigma}} \theta,
\end{align*}
where the last formula follows from a general commutativity fact for inflation and induction ($\Ind_{HN}^G\Inf_{HN/N}^{HN} \chi = \Inf_{G/N}^G \Ind_{HN/N}^{G/N} \chi$ for an abstract group $G$, a subgroup $H\subseteq G$, a normal subgroup $N \subseteq G$, and a representation $\chi$ of $HN/N$). Thus $R_{{\bf x},T,U,s}^{\theta}$ is bigger than $\Inf_{\breve G_r^\sigma}^{\breve G_s^\sigma} R_{{\bf x},T,U,r}^{\theta}$.



\section{Traces of very regular elements}
Let the notation be as in the beginning of Section \ref{sec:Sigma}. 

\begin{Def}\label{def:veryRegularElements}
We say that $s \in \breve P_{\bf x}$ is \emph{unramified very regular} with respect to $\bf{x}$ if the following conditions hold: 
\begin{itemize}
\item[(i)] $s$ is a regular semisimple element of $G_{\breve k}$, 
\item[(ii)] the connected centralizer $Z^\circ(s)$ of $s$ is a $\breve k$-split maximal torus of $G_{\breve k}$ whose apartment contains ${\bf x}$, and
\item[(iii)] $\alpha(s) \not\equiv 1$ modulo $\fp$ for all roots $\alpha$ of $Z^\circ(s)$ in $G_{\breve k}$.
\end{itemize}
For $r\geq 2$, we say that $s \in \bG_r$ is \emph{unramified very regular}, if $s$ is the image of an unramified very regular element of $\breve P_{\bf x}$.
\end{Def}

Note that condition (ii) implies condition (i). Note that in condition (iii) the character $\alpha \colon Z^\circ(s) \rar \bG_{m,\breve k}$ induces a homomorphism of maximal bounded subgroups: $\alpha \colon \breve Z^\circ(s) \rar \cO^\times$, and hence the condition makes sense. 

\begin{rem}
When $G$ is an inner form of $\GL_n$ and $T$ is the maximal nonsplit unramified torus in $G$, Definition \ref{def:veryRegularElements} says that $x \in (\breve T^0)^\sigma = \cO_L^\times$ (here $\breve k \supseteq L \supseteq k$ is the degree-$n$-subextension) is unramified very regular if and only if the image of $x$ in $(\cO_L/U_L^1) \cong \FF_{q^n}^\times$ has trivial $\Gal(\FF_{q^n}/\FF_q)$-stabilizer. This is \textit{not} equivalent to (though is implied by) the condition that the image of $x$ in $\FF_{q^n}^\times$ is a generator although this last condition is sometimes also associated to the same terminology \cite{Henniart_92, BoyarchenkoW_13, CI_ADLV}. \hfill $\Diamond$
\end{rem}

Note that if $s \in \breve P_{\bf x}$ is unramified very regular, then we may consider the $W_{\bf x}(T)$-homogeneous space $W_{\bf x}(T,Z^\circ(s))$ (see Section \ref{sec:Weyl_Bruhat}).

\begin{theorem}\label{thm:traces}
Let $\theta \from \breve T_r^\sigma \to \overline \QQ_\ell^\times$ be any character and let $s \in \breve P_\bx^\sigma$ be unramified very regular with respect to ${\bf x}$. Then
\begin{equation*}
\Tr(g, R_{T,U,r}^\theta) = \sum_{w \in W_{\bx}(T, Z^\circ(g))^\sigma} (\theta \circ {\rm Ad}(w^{-1}))(g).
\end{equation*}
\end{theorem}


\begin{corollary}
Let $T' \subset G$ be a $k$-rational $\breve k$-split maximal torus whose apartment contains $\bx$. If $T$ and $T'$ are not conjugate by an element of $\breve P_{\bf x}^\sigma$, then for any $s \in T'(k)$ unramified very regular with respect to ${\bf x}$,
\begin{equation*}
\Tr(s, R_{T,U,r}^\theta) = 0.
\end{equation*}
\end{corollary}
\begin{proof} We need to show that for two such tori, $W_{\bx}(T, T')^\sigma = \varnothing$. Suppose there is an element $w \in W_{\bx}(T, T')^\sigma$. Then its preimage in $N_{\bG_r}(\bT_r,\bT_r')$ form a $\FF_q$-rational $\bT_r$-torsor, which by Lang's theorem has a rational point. Doing this for all $r$ and using that the inverse limit of a family of non-empty compact sets is non-empty, we can find an element $n \in \breve P_{\bf x}^\sigma$, which conjugates $T(\cO)$ into $T'(\cO)$. The centralizer of $T(\cO)$ in $G(\breve k)$ is $T(\breve k)$ (and similarly for $T'$), so $n$ also conjugates $T(\breve k)$ into $T'(\breve k)$, and so it conjugates $T$ into $T'$, which contradicts the assumption. 
\end{proof}

Let $B$ denote the Borel subgroup of $G$ whose unipotent radical is the fixed subgroup $U$, and let $\bB_r$ be the corresponding subgroup of $\bG_r$. The following result shows that $\bB_r$ behaves in certain aspects like a Borel subgroup of $\bG_r$ (although it is not a Borel subgroup if $r \geq 2$). Similar results in the case that $P_{\bf x}$ is reductive are shown in \cite{Stasinski_12}.

\begin{proposition}\label{p:Br}
Let $s \in \breve G_r$ be an unramified very regular element. If $x \in \breve G_r$ is such that $s \in x \breve B_r x^{-1}$, then there exists a unique $w \in W_{\bx}(T, Z^0(s))$ such that for any lift $\dot w \in \breve G_r$, we have $x \in \dot w \breve B_r$.
\end{proposition}

\begin{proof}
The maximal $\breve k$-split tori $T$ and $Z^\circ(s)$ are conjugate by an element $y \in \breve P_{\bf x}$, as $\bf x$ is contained in the intersection of their apartments. Conjugating by $y$ we thus may reduce to the special case that $Z^\circ(s) = T$. 

We first prove the assertion in the case $r = 1$. The image of $s$ in the reductive group $\breve G_1$ is regular semisimple and $\bB_1 \subseteq \bG_1$ is a Borel subgroup. By \cite[Proposition 4.4(ii)]{DeligneL_76}, we see that there is an element $\dot{w} \in \breve G_1$ normalizing $T$, and satisfying $x \bB_1 x^{-1} = \dot{w}^{-1} \bB_1 \dot{w}$. By the normalizer theorem, $\dot{w}^{-1}x \in \breve B_1$, and we are done.

We now prove the assertion for $r \geq 2$. By the above, we see that there exists a unique $w \in W_\bx(T)$ such that $x \in \dot w \breve B_r \breve G_r^1$. We proceed by induction; to this end, it suffices to prove that if $x \in \dot w \breve B_r \breve G_r^{r-1}$, then $x \in \dot w \breve B_r$.

Since $\bG_r^{r-1}$ is normal in $\bG_r$, we may write $x = \dot w h b$ for some $h \in \breve G_r^{r-1}$ and $b \in \breve B_r$. By \cite[Theorem 4.2]{MoyP_96}, $\breve G_r^{r-1}$ has an Iwahori decomposition, so we may write $h = h_- h_+$ with $h_- \in \breve U_r^{-,r-1}$ and $h_+ \in \breve B_r^{r-1}$. Replacing $b$ by $h_+ b$ and $h$ by $h_-$, we now have $h \in \breve U_r^{-,r-1}$. Since $x^{-1} s x \in \breve B_r$ by assumption, we have $h^{-1}{\rm Ad}(w^{-1})(s)h \in \breve B_r$. Writing $t$ for the very regular element ${\rm Ad}(w^{-1})(s) \in \breve T_r$, we deduce $h^{-1}(tht^{-1})t \in \breve B_r$, and thus $h^{-1}(tht^{-1}) \in \breve B_r$. Since $h \in \breve U_r^{-, r-1}$ by construction, $h^{-1}(t h t^{-1}) \in \breve B_r$ only if $h = t h t^{-1}$, which holds only when $h = 1$ by Lemma \ref{lm:property_of_unram_very_regulars}.
\end{proof}

\begin{lm}\label{lm:property_of_unram_very_regulars}
Let $r\geq 2$ and let $t \in \breve T_r \subset \breve G_r$ be unramified very regular. If $tht^{-1} = h$ for some $h \in \breve U_r$, then $h = 1$. 
\end{lm}
\begin{proof}
Fixing an order on the roots $\Phi^+ = \Phi(T,U)$, we may write $h$ \emph{uniquely} as $\prod_{\alpha\in \Phi^+}\psi_{\alpha}(h_{\alpha})$, where $\psi_{\alpha}$ is an isomorphism of $\bU_{\alpha,r}$ with a framing object coming from the Chevalley system. Then $\prod_{\alpha\in \Phi^+}\psi_{\alpha}(h_{\alpha}) = h = \zeta^{-1}h\zeta = \prod_{\alpha\in \Phi^+}\psi_{\alpha}(\alpha(\zeta^{-1}) h_{\alpha})$, and hence (by uniqueness of the presentation as a product) $h_{\alpha} = \alpha(\zeta^{-1}) h_{\alpha}$. We have naturally $h_\alpha \in \fp^{r_{1,\alpha}}/\fp^{r_{2,\alpha}}$ for appropriate $r_{1,\alpha} \leq r_{2,\alpha} \in \bZ$. As $\zeta^{-1}$ is very regular, $\alpha(\zeta^{-1}) \not\equiv 1 \mod \fp$, and hence the above equality forces $h_{\alpha} = 0$ for all $\alpha$. Thus $h = 1$.
\end{proof}

By Proposition \ref{p:Br},
\begin{equation}
\label{e:gT fixed comp}
S_{T,U}^{(g, \bT)} 
\colonequals \{x \in \breve G_r : \text{$x^{-1} \sigma(x) \in \breve U_r$ and $g x \in x \breve T_r^\sigma$}\}
= \bigsqcup_{w \in W_\bx(T, Z^0(g))} S_{T,U}^{(g, \bT)}(w),
\end{equation}
where
\[
S_{T,U}^{(g, \bT)}(w) \colonequals \{x \in \ddot w \breve B_r : \text{$x^{-1} \sigma(x) \in \breve U_r$ and $g x \in x \breve T_r^\sigma$}\},
\]
for some (any) lift $\ddot{w} \in \bG_r$ of $w$. For any $k$-rational $\breve k$-split maximal torus $T' \subset G$ whose apartment contains ${\bf x}$, the preimage of any $w \in W_{\bf x}(T, T')^\sigma$ in $\bG_r$ is an $\FF_q$-rational $\bT_r$-torsor, so by Lang's theorem, it contains a $\FF_q$-rational point $\dot w$. For any $w \in W_{\bf}(T, T')^\sigma$ we fix such a $\dot w$. 

\begin{proposition}\label{p:gT fixed}
Let $g \in \breve G_r^\sigma$ be a very regular element. Then
\begin{equation*}
S_{T,U}^{(g, \bT)}(w) = \begin{cases} \dot w \breve T_r^\sigma & \text{if $w \in W_\bx(T, Z^0(g))^\sigma$,} \\ \varnothing & \text{otherwise}. \end{cases}
\end{equation*}
\end{proposition}

\enlargethispage*{\baselineskip}
\begin{proof}
Let $w \in W_\bx(T, Z^0(g))$ and let $\ddot w$ be any lift of $w$ to $\breve G_r$. Assume that $S_{T,U}^{(g, \bT)}(w) \neq \varnothing$ and let $x \in S_{T,U}^{(g, \bT)}(w)$. Then $\ddot w^{-1}x \in \breve B_r$. After modifying the lift $\ddot w$, we may write $x = \ddot w v$ with $v \in \breve U_r$. As $(\ddot w v)^{-1}\sigma(\ddot w v) = x^{-1}\sigma(x) \in \breve U_r$, we deduce that $\ddot w^{-1}\sigma(\ddot w) \in \breve U_r$. Projecting to $\bG_1$, we see that $\overline{\ddot w}^{-1}\sigma(\overline{\ddot w}) \in \breve U_1$. The left hand side is a semisimple element in $\bG_1$, and so we deduce that $w \in W_{\bf x}(T)^\sigma$. 

We may now choose $\ddot w$ to be the $\sigma$-stable lift $\dot w$. Then we can write $x = \dot w a v$ where $a \in \breve T_r$ and $v \in \breve U_r$. Then $x^{-1} \sigma(x) = v^{-1} a^{-1} \dot w^{-1} \sigma(\dot w) \sigma(a) \sigma(v) = v^{-1} a^{-1} \sigma(a) \sigma(v) \in \breve U_r$ and hence we must have $a \in \breve T_r^\sigma$. Furthermore, by assumption we have $x^{-1} g x = v^{-1} a^{-1} \dot w^{-1} g \dot w a v = v_1^{-1} \dot w^{-1} g \dot w v_1 \in \breve T_r^\sigma$, where $v_1 \colonequals ava^{-1} \in \breve U_r$. The element $t \colonequals \dot w^{-1} g \dot w \in \breve T_r^\sigma$ is very regular and now $v^{-1} (t v t^{-1}) \in \breve T_r^\sigma$. Hence necessarily $v = t v t^{-1}$, which forces $v = 1$ by Lemma \ref{lm:property_of_unram_very_regulars}.
\end{proof}

\begin{proof}[Proof of Theorem \ref{thm:traces}] 

For any $\breve k$-split maximal torus $T' \subset G$, we have a short exact sequence
\[
1 \rar (\breve T_r'{}^1)^\sigma \rar \breve T_r'{}^\sigma \rar \breve T_1'{}^\sigma \rar 1
\]
of finite abelian groups with $(\breve T_r'{}^1)^\sigma$ of $p$-power order and $\breve T_1'{}^\sigma$ of order prime to $p$. (The surjectivity on the right holds as $\breve T_1'{}^\sigma \rar H^1(\Gal(\overline \FF_q/\FF_q), \breve T_r'{}^1)$ must be the zero morphism, as the latter is a $p$-group). This sequence is split.  
    
Applying the above to $T' = Z^0(g)$, we may write $g = s t_1$ where $t_1 \in (\breve T_r'{}^1)^\sigma$ has $p$-power order and $s$ is in the image of the splitting and hence of order prime to $p$. It is easy to see that $t_1$ and $s$ are both powers of $g$. Note that $s$ is still very regular and $Z^0(s) = Z^0(g)$. Analogously, applying the above to $T' = T$, for any $\tau \in \breve T_r^\sigma$, we may write $\tau = \zeta \tau_1$ with $\tau_1 \in (\breve T_r^1)^\sigma$, and $\zeta$ in the image of the splitting. Thus $(g, \tau) \in \breve G_r^\sigma \times \breve T_r^\sigma$ has the decomposition
\begin{equation*}
(g, \tau) = (s, \zeta) \cdot (t_1, \tau_1),
\end{equation*}
where $(s,\zeta)$ and $(t_1, \tau_1)$ are both powers of $(g,\tau)$ such that $(s,\zeta)$ has prime-to-$p$ order and $(t_1,\tau_1)$ has $p$-power order. Averaging over $\breve T_r^\sigma$ and applying the Deligne--Lusztig trace formula \cite[Theorem 3.2]{DeligneL_76} (which we may do by Lemma \ref{lm:STU_prop}), we deduce
\begin{align}
\nonumber\Tr(g, R_{T,U,r}^\theta) &= \frac{1}{\# \breve T_r^\sigma} \sum_{\tau \in \breve T_r^\sigma} \theta(\tau)^{-1} \Tr\left((g,\tau)^\ast; \textstyle\sum\limits_i(-1)^i H_c^i(S_{T,U}, \cool)\right) \\
\label{eq:comp_traces_DL_formula} &= \frac{1}{\# \breve T_r^\sigma} \sum_{\tau \in \breve T_r^\sigma} \theta(\tau)^{-1} \Tr\left((t_1,\tau_1)^\ast; \textstyle\sum\limits_i (-1)^i H_c^i(S_{T,U}^{(s,\zeta)}, \cool)\right), 
\end{align}
where $S_{T,U}^{(s,\zeta)} \colonequals \{x \in \bG_r \colon x^{-1}\sigma(x) \in \bU_r, sx\zeta = x\}$ is the set of fixed points of $S_{T,U}$ under $(s,\zeta)$.

We obviously have $S_{T,U}^{(s,\zeta)} \subseteq S_{T,U}^{(g, \bT)}$, and it now follows easily from Proposition \ref{p:gT fixed} that 
\begin{equation*}
S_{T,U}^{(s, \zeta)} = 
\begin{cases}
\dot w \bT_r^\sigma & \text{if $\zeta = {\rm Ad}(w^{-1})(s^{-1})$ for some (unique) $w \in W_{\bx}(T, Z^0(g))^\sigma$}, \\
\varnothing & \text{otherwise.}
\end{cases}
\end{equation*}
Now $(t_1,\tau_1)$ acts on a point $\dot w a \in \dot w \bT_r^\sigma$ by $(t_1,\tau_1) \colon \dot w a \mapsto t_1 \dot w a \tau_1 = \dot w {\rm Ad}(w^{-1})(t_1) a \tau_1$, and thus 
\begin{align*}
\Tr\left((t_1, \tau_1)^* ; \textstyle \sum\limits_i (-1)^i H_c^i(S_{T,U}^{(s, \zeta)}, \overline \QQ_\ell)\right) 
&= \Tr\left((t_1, \tau_1)^* ; H_c^0(\dot w \bT_r^\sigma)\right) \\
&= 
\begin{cases}
\# \breve T_r^\sigma & \text{if $\tau_1 = {\rm Ad}(w^{-1})(t_1^{-1})$}, \\
0 & \text{otherwise,}
\end{cases}
\end{align*}
and Theorem \ref{thm:traces} now follows from \eqref{eq:comp_traces_DL_formula}.
\end{proof}

\bibliography{bib_ADLV_CC}{}
\bibliographystyle{alpha}

\end{document}